\documentclass[11pt]{amsart}

\usepackage{color}

\usepackage[english]{babel}
\usepackage[latin1]{inputenc}
\usepackage{lmodern}
\usepackage{float}
\usepackage[T1]{fontenc}

\usepackage{amsmath,amssymb,amsthm}
\usepackage{amscd,accents,array,url}

\usepackage{graphicx}

\usepackage{latexsym}
\usepackage{mathrsfs}

\newtheorem{thm}{Theorem}[section]
\newtheorem{prop}[thm]{Proposition}
\newtheorem{lem}[thm]{Lemma}
\newtheorem{cor}[thm]{Corollary}

\newtheorem{remark}{Remark}

\newcommand\dO{{\rm d}\Omega}

\newcommand\E{{\mathcal P}}

\newcommand\T{{ T}}
\newcommand\Po{{\mathbb P}}
\newcommand\R{{\mathbb R}}
\newcommand\I{{\mathbb I}}
\newcommand\PP{{\mathcal P}}
\renewcommand\S{{\mathcal S}}

\newcommand\nn{{\bf n}}

\newcommand\calB{{\mathcal B}}
\newcommand\calO{{\mathcal O}}
\newcommand\calV{{\mathcal V}}
\newcommand\kd{k_\Delta}
\newcommand\ecurvo{\gamma}
\newcommand\Th{{\mathcal T}_h}
\newcommand\Ro{\rho}

\newcommand\uh{u_h}
\newcommand\uI{u_I}

\newcommand\ghD{\overline{g_D}}

\newcommand\up{u_{\pi}}

\newcommand\tri{|\!|\!|}
\newcommand\eps{\varepsilon}
\newcommand\calP{{\mathcal P}}
\newcommand\varphib{\boldsymbol{\varphi}}

\newcommand\pippof{{\kappa}}
\newcommand\pippoi{{\kappa_i}}
\newcommand\pippoz{{\kappa_0}}
\newcommand\gu{{\textsc{u}}}
\newcommand\gv{{\textsc{v}}}
\newcommand\gw{{\textsc{w}}}

\newcommand\fu{\hat{\gu}}
\newcommand\fuI{\fu_I}
\newcommand\fv{\hat{\gv}}
\newcommand\fw{\hat{\gw}}
\newcommand\fuh{{\hat{\gu}_h}}

\newcommand\bdelta{{\mbox{\small $\Theta$}}}

\newcommand\bdeltah{\hat{\bdelta}}
\newcommand\deltas{{{\theta}}}
\newcommand\newthetas{{\chi}}

\begin{document}

\markboth{Beir\~ao da Veiga, Brezzi, Cangiani, Manzini, Marini, Russo}{Basic Principles of Virtual Element Methods}

\title{Virtual Elements and curved edges}

\author{L. BEIR\~AO da VEIGA$^{1,2}$}

\thanks{$^1$Dipartimento di Matematica e Applicazioni, Universit\`a di Milano--Bicocca, Via Cozzi 53, I-20153, Milano 
(Italy); lourenco.beirao@unimib.it, alessandro.russo@unimib.it}

\author{F. BREZZI$^{2}$}

\thanks{$^2$IMATI-CNR, Via Ferrata 1, 27100 Pavia (Italy); brezzi@imati.cnr.it}

\author{L.D. MARINI$^{3,2}$}

\thanks{$^3$Dipartimento di Matematica, Universit\`a di Pavia,
Via Ferrata 1, 27100 Pavia (Italy); marini@imati.cnr.it}

\author{A. RUSSO$^{1,2}$}

\maketitle

%
%
%
%

\section{Introduction}
In this paper we tackle the problem of constructing {\it conforming} Virtual Element spaces on polygons with curved edges.  Apart from the obvious convenience in treating computational domains with curved boundaries with a better accuracy, the interest in using polygons with curved edges could arise in various circumstances,
as when a singularity of the solution near the boundary is expected (e.g. in Fig.~\ref{Anelli2}), or in the presence 
of mixtures of different materials or rough boundaries (as in 
Fig.~\ref{Veletette}) when we are still far from the homogenized limit, and so on.

The treatment of curved edges with {\it nonconforming} approximations
(as in non-conforming VEMs and HHO methods, see \cite{VEMnc}, \cite{HHO-curved}) is relatively easier,
since the traces of the elements of the Virtual Element spaces do not need to be continuous at the endpoints of edges (and in particular at the endpoints of the {\it curved} edges).

For the sake of simplicity here we will just treat the case of a {\it single curved edge per element}, but the general philosophy can be easily applied to the case of several curved edges. Moreover we assume that the curved edge is given in parametric form (say: $x=x(t)$, $y=y(t)$ for $t\in [t_0, t_1]$).                                          

Conforming Virtual Element spaces on polygons with curved edges (still given in parametric form) have already been introduced in \cite{rozzi}, using, on the curved edges, functions that are {\it polynomials in the parameter{ $t$}}. Such an approach has several advantages (including the simplicity of the definition and the adaptability to more general situations).

On the other hand, 
one of the advantages of the present approach, compared with
others like the one in \cite{rozzi}, is that the local VEM spaces contain all the polynomials of a given degree (and this provides the {\it patch test}). Moreover the result will not depend (apart from round-off errors) from the choice of the parametrization of the curve.

For another different approach and use of Virtual Elements with curved edges we refer also to \cite{VEM-curv-Bertol}, where the authors consider a polygonal approximation $\Omega_h$ (with straight edges) of the original domain $\Omega$ (with curved edges). The discrete solution (say, $u_h$) is originally defined in $\Omega_h$. But then a smart correction $\tilde{u}_h$ (using a suitable approximation of high order normal derivatives of $u_h$) 
is introduced, and used to define a problem in $\Omega$ (very much in the spirit of \cite{Bramble-Dupont-Thomee}, and then \cite{Bramble-King94},  \cite{Bramble-King97}).

Another approach to deal with conforming Virtual Elements with  curved edges has been proposed in \cite{Ovall}. This last approach is very similar to the present one (although the two have been developed independently). The (minor!) difference between the two treatments is tricky, and we describe it in detail later on, in Remark \ref{cfrovall}: we anticipate here that the one in  in \cite{Ovall} is somehow {\it simpler}, while we think that the one presented here is {\it safer} ( = less exposed to the possible damages of round-off errors).

We also point out that more traditional ways of dealing with curved edges in Finite Elements (typically, using iso-parametric elements) are based on the property that: {\it we have a fixed curve
that tends to be (locally!) straight when the discretization gets finer and finer}. Here we want to be able, as well, to work somewhat {\it in between}: when the discontinuities in the material properties (or the roughness of the boundaries) is still far from the homogenized limit but the curved boundaries are far from being locally flat. A numerical example in this direction will be given in the last section.

\begin{figure}[!htb]
\begin{center}
\includegraphics[height=0.45\textwidth]{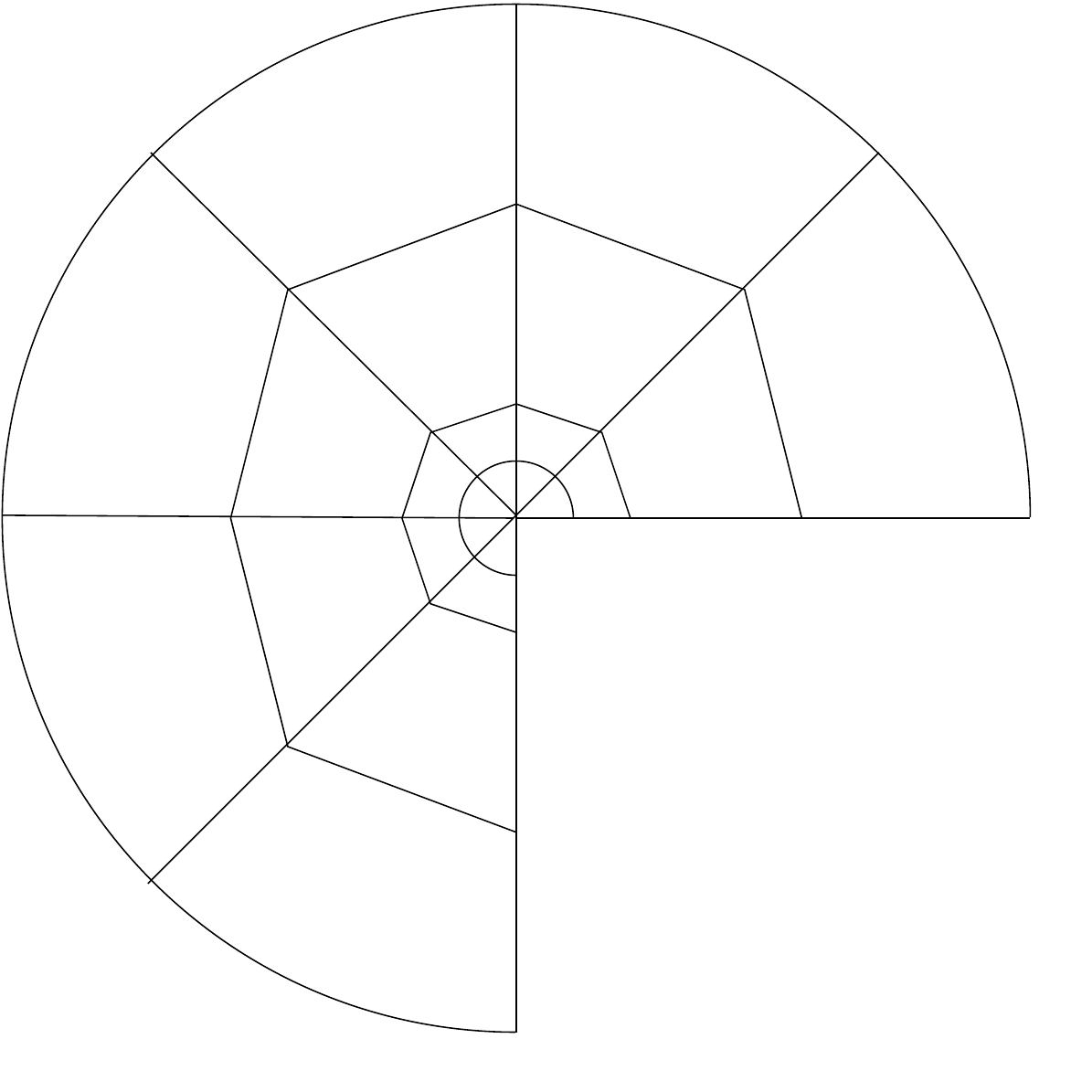} 
\end{center}
\caption{Taking into account a re-entrant corner}
\label{Anelli2}
\end{figure}

\begin{figure}[!htb]
\begin{center}
\includegraphics[height=0.35\textwidth]{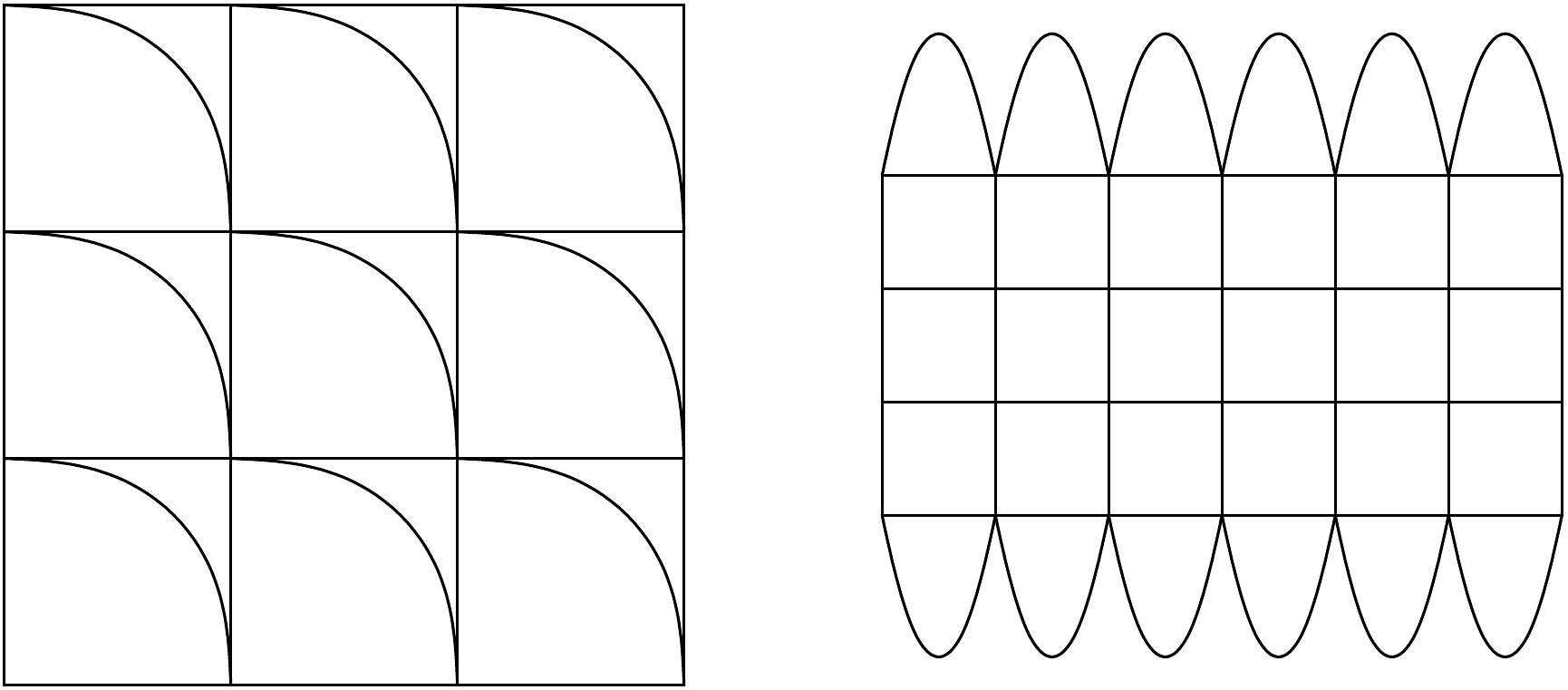} 
\end{center}
\caption{Other decompositions}
\label{Veletette}
\end{figure}

 In this presentation, it will be convenient to  distinguish  different  types of curved edges: curved edges internal to the computational domain (typically, to be used in the presence of two different materials separated by a curved interface), curved edges belonging to a part of the boundary where Dirichlet conditions are imposed, and  
curved edges belonging to a part of the boundary where Neumann or Robin boundary conditions are imposed.


We point out that we are ready to accept that an edge  which a-priori is declared to be be curved might, by chance, be actually straight. On the other hand there might be parts of the boundaries that are "known to be straight", and these will be treated as straight from the very beginning. Hence, to summarize, we will assume that there are: {\it i)} edges that are declared as "straight" (and {\bf must} be straight), and  {\it ii)} edges that are declared as "curved" (and {\bf might} be either curved or straight). 

An outline of the paper is as follows: In the next Section \ref{model} we will start  recalling some basic notation on polynomial spaces, and then we will present the model problem to be used in the following discussion: essentially, 
the simple linear elliptic problem $-{\rm div}(\pippof\nabla u)=f$   in a domain $\Omega$, with Dirichlet boundary conditions on one part of the boundary and Robin boundary condition on the remaining part. We will also assume, for simplicity, that the material coefficient $\pippof$ is piece-wise constant, assuming one constant value in a sub-domain $\Omega_i\subset\subset\Omega$ and another constant value in the remaining part.  In Section \ref{locspa} we will introduce the Virtual Element local spaces to be used in the sub-domains of the decomposition. In the following Sections \ref{locstiff} and \ref{glostiff} we will then introduce the {\it local} stiffness matrices and the {\it global} ones, respectively.  Finally, in Section \ref{dispro} we will introduce the discretized problem and present the corresponding bounds of the error in terms of the interpolation errors that, in turn, will be discussed in Section \ref{interr}.


The whole treatment here has been concentrated on simple linear elliptic equations in 2 dimensions. Needless to say, there is a lot of interesting work still to be done, in order to extend it to more complex problems, much more interesting for applications.
This, apart from the natural need to deal with the three-dimensional case, includes obviously all the subjects where Virtual Elements proved to be a viable and useful discretization method, as for instance: linear elasticity problems, \cite{VEM-elasticity} \cite{HR-Virtual-Plate}; incompressible and nearly incompressible materials, \cite{ABMV14}, \cite{NewStokes}, \cite{Stokes:divfree};
    plate bending and more generally fourth order problems,
\cite{Brezzi:Marini:plates}, \cite{ABSVu}
   \cite{cinesi_plates}; electromagnetic problems 
\cite{lowest-max3}, \cite{max3} and wave propagation, 
\cite{Acoustic},  \cite{Helmo-PPR},\cite{Acoustic},
and so on. Moreover, the use of elements with curved boundary to deal with contact problems (already treated successfully for VEM discretizations with straight edges e.g. in \cite{Wriggers-1}, \cite{Wriggers-2} ) is a surely challenging topic full of potential troubles and appealing perspectives.

Similarly in mesh adaptation for dealing with singularity     (as  in  \cite{Arxiv-hp-corner}), the use of curved edges could also be very attractive.

On the other hand, more theoretical aspects could also be profitably extended, as the use of $H({\rm div})$ and $H({\rm curl})$ elements
(see e.g. {super-misti},as well as applications to problems with variable coefficients (as e.g.in \cite{variable-primal}),
     or the finer aspects of interpolation errors (see \cite{Brenner}, \cite{BLR-stab}).
 
Finally, the Serendipity versions (see e.g. \cite{SERE-nod}, \cite{SERE-mix}) of all these Virtual Element spaces is possibly the most natural extension of the present work, and the basic ideas for doing that are sketched in Remark \ref{XSere}.

\section{The continuous problem}\label{model}

\subsection{Notation} Throughout the paper, if $d$ (= {\it dimension}) is an integer $\ge 1$ and $k$ is an integer, we will denote by 
$\Po_{k,d}$ the space of polynomials of degree $\le k$ in $d$
dimensions. As usual, $\Po_{-1}=0$.
In practice, we  will consider here only the cases
$d=1$ and $d=2$.
We will denote by $\pi^d_k$ the dimension of $\Po_k$ in 
$d$ dimensions, so that for $k\ge 0$ we have
\begin{equation}
\pi^1_k=k+1\qquad\mbox{ and }\qquad\pi^2_k=(k+1)(k+2)/2 .
\end{equation}
In most cases, when no confusion can occur, we will just use $\pi_{k}$ instead of $\pi^1_k$ or $\pi^2_k$. 

If $\calO$ is a domain in $\R^d$ and
$k$ is an integer, we will donote by $\Po_k(\calO)$ the space of the restrictions to $\calO$ of the polynomials
of $\Po_k$. With an abuse of notation, we will often use
simply $\Po_k$ instead of $\Po_k(\calO)$ when no confusion
is likely to occur.

For a $1$-d manifold $\Gamma$ we denote 
by $|\Gamma|$ its length. For an open $\omega\subset \R^2$ we  denote by $|\omega|$ its area.
\smallskip

Moreover, we will denote by $\Pi^{0,\calO}_k$ the  orthogonal-projection operator, in $L^2(\calO)$, onto $\Po_k(\calO)$, defined, as usual, for every $v\in L^2(\calO)$, by
\begin{equation}\label{defPros}
\Pi^{0,\calO}_k(v)\in \Po_k(\calO)\qquad\mbox{{\it and }}\qquad
\int_{\calO}(v-\Pi^{0,\calO}_k(v))\,q_k\,{\rm d}\calO=0\quad
\forall q_k \in \Po_k(\calO). 
\end{equation}

Finally, given   a function $\psi\in L^2(\calO)$ and an integer $s\ge 0$, we recall that the {\it moments of order $\le s$ of 
$\psi$ on $\calO$} are defined as: 
\begin{equation}
\int_{\calO}\psi\,q_s\,{\rm d}\calO\qquad\mbox{ for }q_s\in \Po_s(\calO).
\end{equation}
Hence {\it to assign the moments of $\psi$ up to the order $s$ on $\calO$} will amount to $\pi^d_s$ conditions.
 Tipically this will be used when these moments are considered as {\it degrees of freedom}. Then, we will take in $\Po_s$ a basis $\{q_i\}$ such that $\|q_i\|_{L^1} \simeq 1$.

Throughout the paper we will follow the common notation for scalar products, norms, and seminorms. In particular, 
$(v,w)_{0,\calO} \mbox{ (sometimes, just }(v,w)_{0}  )$ and $\|v\|_{0,\calO} \mbox{ (sometimes, just } \|v\|_0)$ will denote the $L^2$ scalar product and norm, $|v|_{1,\calO} \mbox{ (sometimes, just } |v|_1)$ and $\|v\|_{1,\calO} \mbox{ (sometimes, just  } \|v\|_1^2)$ the $H^1$ semi-norm and norm.

%

\subsection{The continuous problem}

Let $\Omega$ be a bounded open subset of $\R^2$ with Lipschitz continuous boundary
and let $\Omega_i$ be a subset of $\Omega$ with a  Lipschitz continuous boundary, such that (for simplicity) $\overline{\Omega_i}\subset
\Omega$ (implying in particular that $\partial\Omega_i$ cannot touch $\partial\Omega$). See Fig.~\ref{Farlo-Deco}.

\begin{figure}[!htb]
\begin{center}
\includegraphics[height=0.30\textwidth]{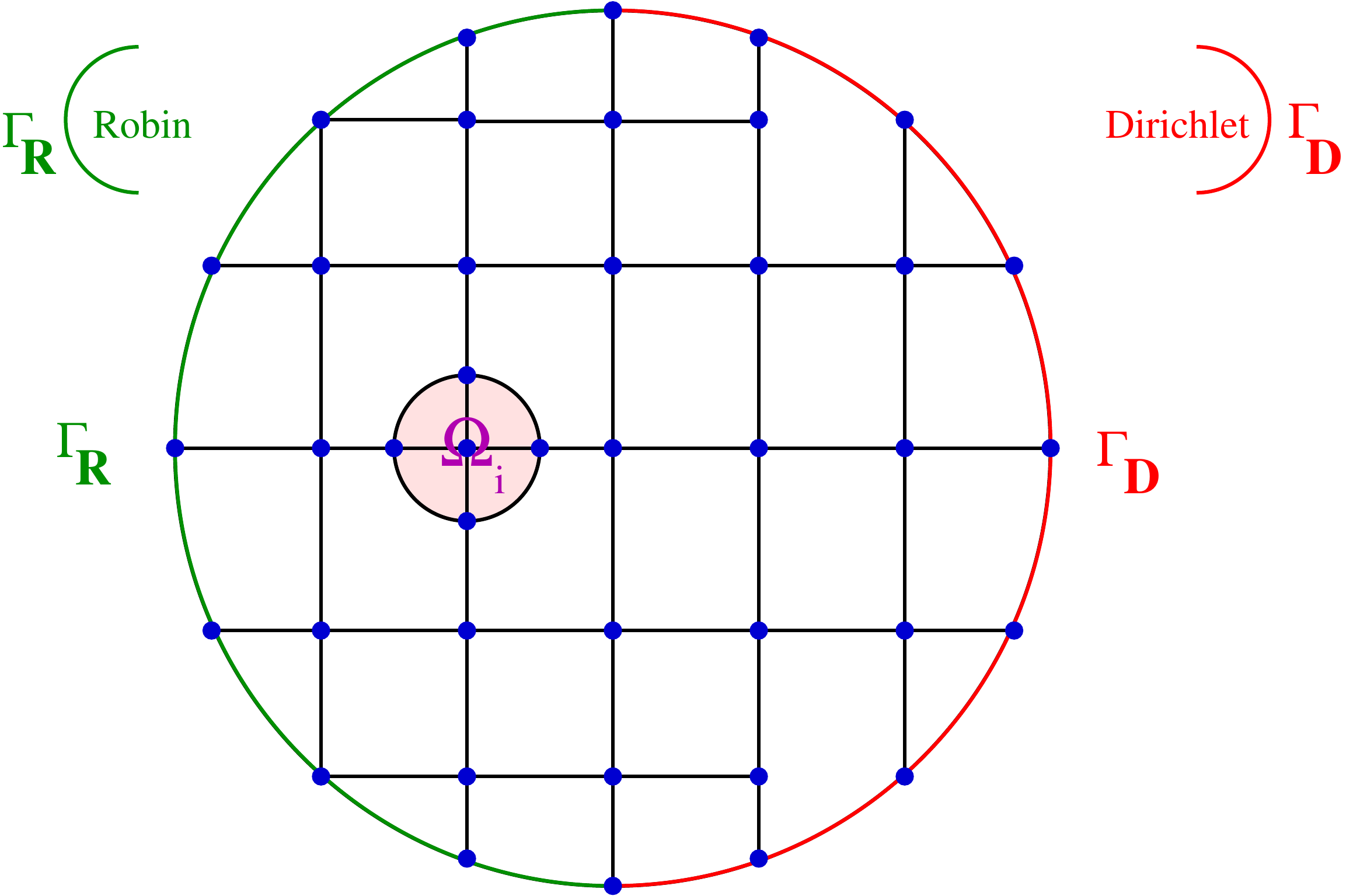} 
\end{center}
\caption{A decomposition with curved edges}
\label{Farlo-Deco}
\end{figure}

We will assume that the material properties in $\Omega_i$ might be different from the ones in the remaining part of $\Omega$. With a suitable attention to regularity issues one could also  easily treat more general cases.

{\color{black}
As already observed in the Introduction, we recall that both on $\partial\Omega_i$ and on 
$\partial\Omega$ we might have edges that are {\color{black} are declared to be straight. In that case it would be natural to treat the corresponding element as a {\it polygonal element}}. On the other hand, it would be cumbersome to check, for every boundary edge, whether it is {\it straight} (or {\it very close to straight}) or {\it curved}. Hence, we will {\bf assume} that each boundary edge belongs to one (and only one) of two distinct sets: in the first we put the ones that we know (from the very beginning) to be straight, and in the second we put the ones that {\it might be either
curved or straight or almost straight}. All the edges (and the corresponding elements) in each group will be treated in the same manner.  This implies that our treatment of curved edges should not fail if by chance the edge is straight or almost straight.  
}

\smallskip
 We consider the model problem
\begin{equation}
\left\{
\begin{aligned}\label{procon}
& \mbox{Find } u\in H^1(\Omega)\mbox{ such that: }\\
&-{\rm div}(\pippof\nabla u)=f\;\mbox{ in }\Omega,\\
&\mbox{with }u=g_D\mbox { on }\Gamma_D\mbox{ and }
(\pippof\nabla u)\cdot{\nn}+\Ro u\,=\,g_R \mbox{ on }\Gamma_R,
\end{aligned}
\right.
\end{equation}
where
\begin{itemize}
\item $\pippof$ assumes two constant values: ${\pippoi}$ in $\Omega_i$ and ${\pippoz}$ in  $\Omega\setminus \Omega_i$, both $>0$, 
\item $\Gamma_D$ and $\Gamma_N$ are open connected subsets of $\partial\Omega$ with $|\Gamma_D|>0$, $\overline{\Gamma}_D\cup\overline{\Gamma}_R=\partial\Omega$ and $\Gamma_D\cap\Gamma_R=\emptyset$,
\item $f$ is given, say, in $L^2(\Omega)$,
\item $g_D$ is given, say, in $H^{1}(\Gamma_D)$,
\item $g_R$ is given, say, in $L^2(\Gamma_D)$,
\item $\Ro\ge 0$ is given in $L^{\infty}(\Gamma_R)$. Note that for $\Ro=0$ we are back to Neumann boundary conditions.
\end{itemize}
It is very well known (and  not difficult to see) that defining
\begin{equation}
H^1_{0,{\Gamma_D}} :=\{v\in H^1(\Omega) \mbox{ such that }
v=0 \mbox{ on }\Gamma_D\},
\end{equation}
and
\begin{equation}
H^1_{{g_D},{\Gamma_D}} :=\{v\in H^1(\Omega) \mbox{ such that }
v=g_D \mbox{ on }\Gamma_D\},
\end{equation}
the solution $u$ of \eqref{procon} coincides with the (unique) solution of the variational problem
\begin{equation}
\left\{
\begin{aligned}\label{proconvar}
& \mbox{find } u\in H^1_{g_D,\Gamma_D} \mbox{ such that }\\
&\int_{\Omega}\pippof\nabla u\cdot\nabla v\,{\rm d}\Omega
+\int_{\Gamma_R}\Ro\,u\,v\,{\rm d}s=\int_{\Omega}f\,v\,{\rm d}{\Omega}
+\int_{\Gamma_R}g_R\,v\,{\rm d}s\;\;
\forall v\in H^1_{0,\Gamma_D}.
\end{aligned}
\right.
\end{equation}
We also point out that, in a natural way, \eqref{proconvar} also implies that the co-normal derivative $\pippof\nabla u\cdot\nn$ is continuous across the boundary of $\Omega_i$.

\subsection{The decomposition}

Let $\Th$ be a decomposition  of $\Omega$ into polygons 
$\PP$. For simplicity of exposition we assume that each polygon has at most one curved edge. We also assume that $\partial\Omega_i$ is all contained in the union of the $\partial\PP$ (in other words: the decomposition respects the discontinuities of $\pippof$). See
Fig. \ref{Farlo-Deco}.

In what follows, we are going to construct virtual element spaces that are suitable for the discretization of  \eqref{proconvar} using decompositions of the type considered above. 
{\bf A priori}, we should be prepared  to distinguish among several different types of elements:
\begin{itemize}
\item $\{\bf 0\}$ Polygons with straight edges. For them we are going to use classical VEMs.
\item $\{\bf 1\}$ Elements with a curved edge shared with  another element. For instance, 
referring  to Fig.\ref{Farlo-Deco}, the elements having a curved edge that belongs to $\partial\Omega_i$. There are 8 of them in  Fig. \ref{fig:bordi}.
\item  $\{\bf 2\}$ Elements that have a curved edge that belongs entirely to $\overline{\Gamma}_D$. There are 8 of them in our Fig. \ref{fig:bordi}.
\item $\{\bf 3\}$  Elements that have a curved edge that belongs entirely to $\overline{\Gamma}_R$. There are 8 of them in our 
 Fig. \ref{fig:bordi}.  
\end{itemize}
\begin{figure}[!htb]
\begin{center}
\includegraphics[height=0.35\textwidth]{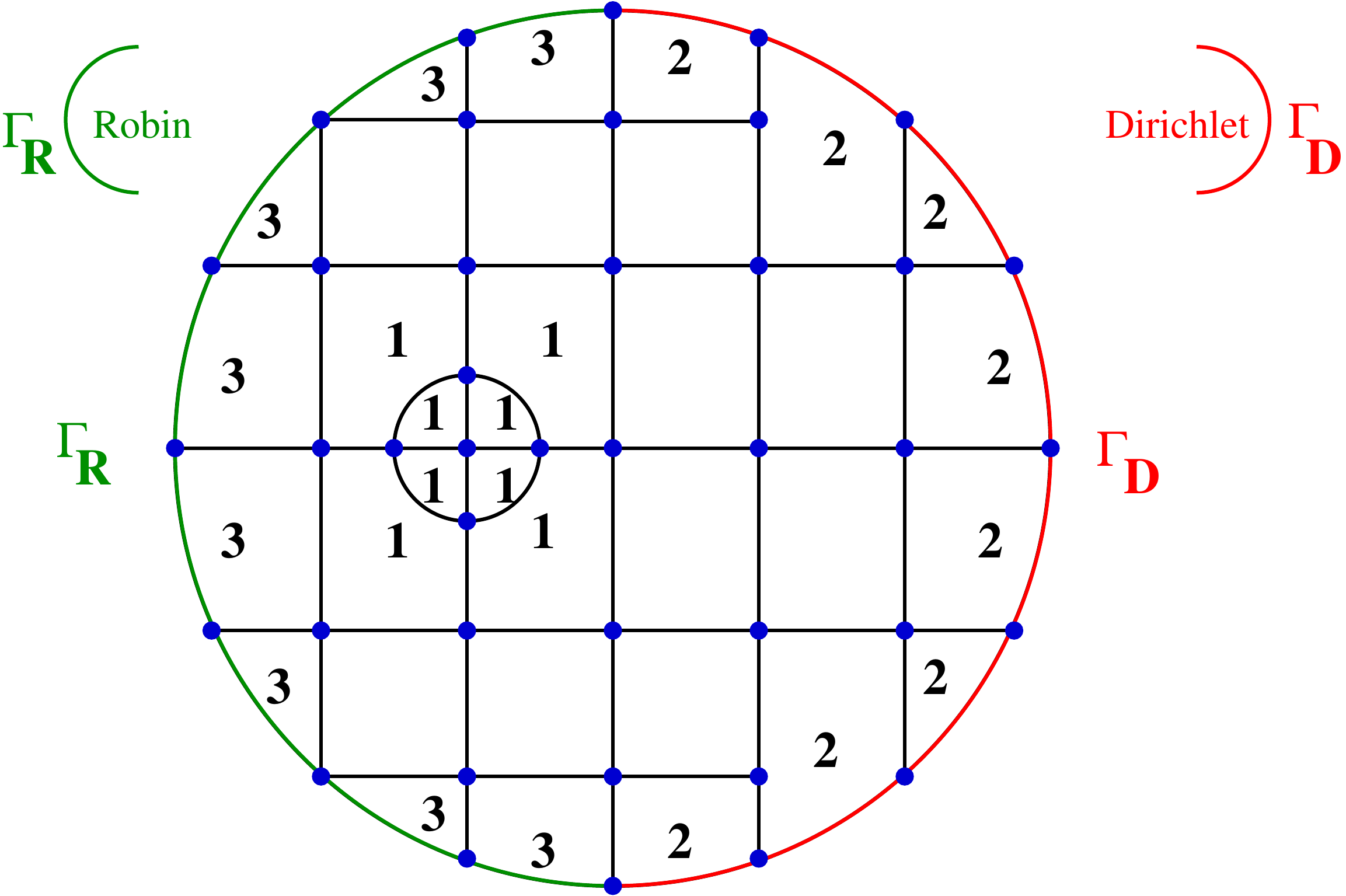} 
\end{center}
\caption{Types of elements with a curved edge}
\label{fig:bordi}
\end{figure}
%



\section{The local spaces}\label{locspa}
\subsection{Subspaces on polygons with straight edges}
We { start} by recalling the classical VEMs commonly used in polygonal decompositions.   
%
For a given integer $k\ge 1$ we consider the trace space
\begin{equation} 
\calB_k(\partial\PP):=\{v\in C^0(\partial\PP)\mbox{ such that }
v_{|e}\in\Po_k(e)\;\forall\mbox{ edge }e\subset{\partial\PP}\}.
\end{equation}
Then, for another integer $\kd$ with $-1\le \kd\le k$ we consider the space:
\begin{equation} 
\calV_{k,\kd}(\PP):=\{v\mbox{ such that }v_{|\partial\PP}\in
\calB_k(\partial\PP), \mbox{ and }\Delta v\in\Po_{\kd}(\PP)\}.
\end{equation}
The natural set of degrees of freedom for 
$\calV_{k,\kd}(\PP)$ is given by (see e.g. \cite{projectors})
\begin{itemize}
\item the values at each vertex of $\PP$,
\item (for $k\ge 2$) the values at the $k-1$ Gauss-Lobatto points of each edge of $\PP$,
\item (for $\kd\ge 0$) the moments of order $\le \kd$ inside $\PP$. 
\end{itemize}
We point out that polygons having {\it two or more consecutive 
edges that belong to the same straight line} are perfectly
allowed, so that {\it vertices and edges}, here, do not coincide with the naive idea that one might have when speaking of {\it vertices and edges of a polygon}. See the example in Fig. \ref{fig:exa}.
\begin{figure}[!h]
\begin{center}
\includegraphics[height=0.250\textwidth]{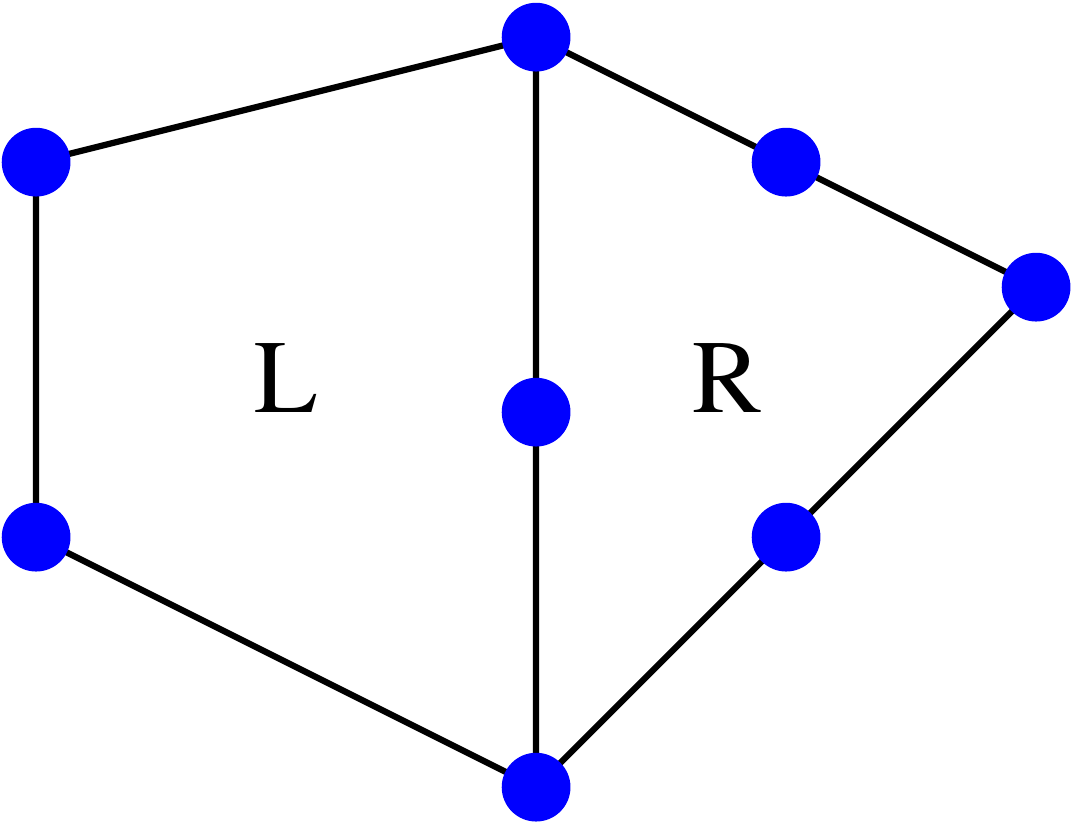} 
\end{center}
\caption{The Left element has 5 vertices and edges; the Right element has 6 vertices and edges.}
\label{fig:exa}
\end{figure}

Here, for simplicity, we will stick on the simplest case where $\kd\equiv k-2$, which is the original choice of \cite{volley}. Hence for $k\ge 1$ we set
\begin{equation}\label{def-Vk}
\calV_k(\PP):=\calV_{k,k-2}(\PP)
\end{equation}
and we have therefore the degrees of freedom:
\begin{equation}\label{dof-Vk}
\begin{aligned}
&\mbox{- the values at each vertex of $\PP$},\\
&\mbox{-  (for $k\ge 2$) the values at the $k-1$ Gauss-Lobatto points of each edge of $\PP$},\\
&\mbox{- (for $k\ge 2$) the moments of order $\le k-2$ inside $\PP$} 
\end{aligned}
\end{equation}
\begin{remark}\label{XSere} In previous works we used the bigger spaces
(corresponding to $\kd=k-1$ or $\kd=k$) as starting point  for a Serendipity correction (see \cite{SERE-nod}) that allowed to end up with local VEM spaces smaller than the $\calV_k$ defined in \eqref{def-Vk}. Here however the Serendipity procedure should be done in different ways for different types of elements (according to our previous classification), and the presentation would become more cumbersome.
\qed
\end{remark}

\subsection{Subspaces on elements with a curved edge not in $\Gamma_D$}

When dealing with elements $\PP$ having a curved edge, we would like to 
follow the same philosophy of the straight polygons. The main feature that we are not willing to give-up  here is the fact that the polynomial space 
$\Po_k(\PP)$ is included in the VEM space ${\calV}_k(\PP)$.

As we already said, we are ready to accept that an edge 
which a-priori could be curved (say, an edge in $\partial\Omega_i$ or in $\partial\Omega$) might, by chance, be straight. 
However, in the present theoretical treatment we will assume, for simplicity, that
all the edges in $\partial\Omega_i$ and in $\partial\Omega$
are declared as "curved", and all the other internal edges are declared as "straight".  We recall that, for simplicity of exposition,  we also assumed that each element has at most one curved edge.  

 To start with, we define therefore the space of traces on a curved edge as follows.  For a given integer $k\ge 1$, on a given element 
$\PP$ with a curved edge $\gamma$, we consider the trace space
\begin{align*} 
\calB_k(\partial\PP)&:=\{v\in C^0(\partial\PP)\mbox{ such that: } v_{|e}\in\Po_k(e)\, \forall\mbox{ straight edge }e\subset{\partial\PP}, 
\\
&\hskip1cm\mbox{ and } v_{|\ecurvo}\equiv {q_k}_{|\ecurvo} \mbox{ for some } q_k\in\Po_k \}. 
\end{align*}
It is clear that:
 
\centerline{\it the trace on $\partial\PP$ of every polynomial of $\Po_k$ belongs to  $\calB_k(\partial\PP)$.}

\smallskip
\noindent Now everything seems to be done, and for every integer $k\ge 1$ 
we can define ({\it exactly as before})
\begin{equation}\label{def-calVk-c}
{\calV}_k(\PP):=\{v\mbox{ such that }v_{|\partial\PP}\in
\calB_k(\partial\PP), \mbox{ and }\Delta v\in\Po_{k-2}(\PP)\}.
\end{equation}
Here too it is easy to check that for every $k\ge 1$ we have
\begin{equation}
\Po_k(\PP)\subseteq {\calV}_k(\PP).
\end{equation}
The delicate point comes from the choice of the {\it degrees of freedom}, treated in the next subsection. 

\subsection{VEM spaces, degrees of freedom, and "generators"}

In $\calB_k(\partial\PP)$ it seems natural to start taking as degrees of freedom: 
\begin{itemize}
\item The values of $v$ at the vertices,
\end{itemize}
 and {\it for every straight edge} $e$, and $k\ge 2$: 
\begin{itemize}
\item The values of $v$ at the $k-1$ Gauss-Lobatto points of $e$.
\end{itemize}
With that, we took care of the straight edges {\it and of  the values
at the endpoints of the curved edge}. But already for $k=1$ we need some additional information on the edge $\gamma$ declared as curved. Indeed,  once we know the values of a polynomial in $\Po_k$ at two distinct points, we still need $\pi_k-2$ degrees of freedom to determine it  uniquely. This will be done by considering the values at $\pi_k-2$ {\it fictitious points}, that we call {\it trace generator points} (or, simply, {\it tg points}, or even {\it tgp}) located as in Fig. \ref{fig:tragen}  (essentially: the points that would normally be used to place {\it the degrees of freedom for $\Po_k$ on an ideal 
triangle having "the segment joining the endpoints of the curved edge" as base}).

\begin{figure}[!h]
\begin{center}
\includegraphics[height=0.30\textwidth]{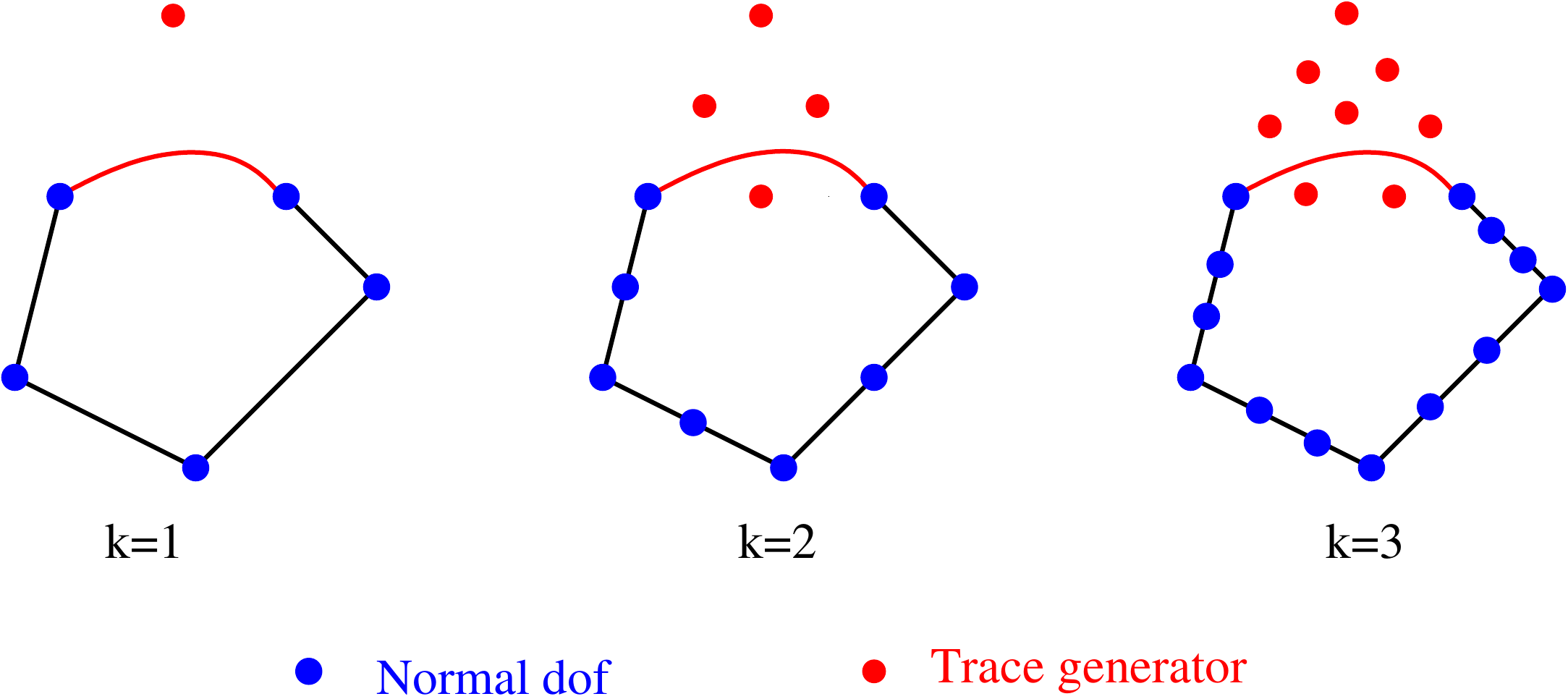} 
\end{center}
\caption{{\color{black} Normal degrees of freedom} and {\color{black}trace generators} for $k=1, 2, 3$}
\label{fig:tragen}
\end{figure}

We consider therefore the set of values:
\begin{equation}\label{pseudodof}
\left\{
\begin{aligned} 
&\bullet\mbox{ the values at each vertex of $\PP$,}\\
&\bullet\mbox{ (for $k\ge 2$) the values at the $k-1$ Gauss-Lobatto points of each 
straight edge of $\PP$},\\
&\bullet\mbox{ the values at the $\pi_k-2$ trace generator points of the curved edge.}
\end{aligned}
\right.
\end{equation}

We point out that the values indicated in \eqref{pseudodof} can identify uniquely an element   of 
$\calB_k(\partial\PP)$
but cannot be taken as {\it degrees of freedom} 
{\it in the classical sense}. Indeed, it is very easy to see that on the curved edge $\gamma$:
\begin{itemize}
\item for every $k\ge 1$,  for every set of $\pi_k$ values (at the  trace generators plus the two endpoints), there exists, on $\gamma$, a unique function  which is ``the restriction to $\gamma$ of a polynomial $q_k$ in $\Po_k$  which assumes the given $\pi_k$ values at the $\pi_k$ points".
\item However, depending on the shape of $\gamma$, there might be several different  ${q_k}$'s that have the same restriction  to $\gamma$ (see the example here below).
\end{itemize} 
\noindent
Indeed, already for $k=1$, if $\gamma$ is a straight segment (a case that we {\bf do not} want to forbid!) the restriction of a polynomial of degree $\le 1$ to 
$\gamma$ will depend only on the values of the polynomial at the endpoints (= vertices of $\PP$) and not on its value at the trace generator. So we will have one tgp plus two endpoints, but the dimension of the  space of their restrictions to $\gamma$ will be only 2.


Summarizing: an  element 
$v\in {\calV}_k(\PP)$ {\it could be identified} by
\begin{equation}
\begin{aligned}\label{dof}
&\mbox{- the values at each vertex of $\PP$},\\
&\mbox{- (for $k\ge 2$) the values at the $k-1$ Gauss-Lobatto points of each
straight edge of $\PP$},\\
&\mbox{- the values at the trace generator points of the curved edge $\gamma$},\\
&\mbox{- (for $k\ge 2$) the moments of order $\le k-2$ inside $\PP$}.
\end{aligned}
\end{equation}
However, such an {\it identification} will not be {\it injective}, as different sets of the above parameters \eqref{dof} might generate the same element of ${\calV}_k(\PP)$. Hence it is a natural choice {\it not to call} them {\it degrees of freedom}, and {\it for the quantities \eqref{dof} we are going to stick instead to the name {\bf generator values}, or simply {\bf generators}}.

All this, and the treatment of other couplings $\{space, generators\}$ that we are going to see all over the paper, might induce a certain amount of confusion: not really in the proofs (as we shall see), and even less in the code (once you know what has to be done). But the {\it description of the method} might easily become confused. Indeed, our trial and test variables will have {\it two faces}.
We must consider 
\begin{itemize}
\item The set of generators \eqref{dof}: this is inevitably the only thing that will be seen in the code. We will indicate them (both locally and globally) with {\it small capital} letters: $\gu$, $\gv$, $\gw$, etc.
\item To each set of generators we attach a function that leaves in one of our VEM spaces. These will be indicated by $\fu$, $\fv$, $\fw$, etc.  (or, sometimes, even by $u$, $v$, or $w$), respectively.
\end{itemize}

Occasionally it will also be convenient to introduce
\begin{equation}\label{defbG}
{\bf G}\{{\calV}_k(\PP)\}:=\mbox{ the set of generators of } 
{\calV}_k(\PP)
\end{equation}
that obviously coincides with $\R^{N_k^{\PP}}$, where 
$N_k^{\PP}$ is the dimension of the space of generators of ${\calV}_k(\PP)$.
\begin{remark}\label{fvzero}
It has to be pointed out that for each set of generators (say, $\gv$) we have a unique associated function $\fv$, but {\it the converse is not true}, as we have seen. Indeed, there might be sets of generators $\gv\not={\bf 0}$ whose associated function $\fv$  is identically zero. In such a case the space of generators ${\bf G}\{{\calV}_k(\PP)\}$ would have a dimension {\it bigger} than that of the functional space ${\calV}_k(\PP)$. 
The unknown solution of the discretized problem, in the computer code, will be a set
of generators $\gu_h\in {\bf G}\{{\calV}_k(\PP)\}$. This $\gu_h$ will identify uniquely an element $\fuh$ of our VEM space, corresponding to what, in almost all papers on numerical methods for PDE's, would be denoted by $u_h$.
\hfill\qed
\end{remark}

\begin{remark}\label{Gpol}
 We observe that, whenever a function $v \in {{\calV}_k(\PP)}$ is a polynomial, it is always possible to associate with it a set of generators $\gv\in{\bf G}
\{{\calV}_k(\PP)\}$, defined through \eqref{dof}. We denote this by
\begin{equation}\label{gen-pol}
{\gv}={\mathcal G} v\qquad v\in \Po_k(\PP).
\end{equation}
We underline the fact that the operator ${\mathcal G}$ is {\bf not} defined on the whole space  ${\calV}_k(\PP)$, but only on 
$\Po_k(\PP)\subseteq {\calV}_k(\PP)$. In particular, having  a generic function $v \in {\calV}_k(\PP)$, in order to reconstruct {\bf a} set of generators $\gv$ such that $\fv \equiv v$ we must prescribe the values of $\gv$ on the trace generator points (obviously, among those that generate the right value of $v$ on $\gamma$).
\hfill\qed
\end{remark} 


\begin{remark}\label{cfrovall}
At this point we are able to detail in a more precise way the difference between the present approach and the one in \cite{Ovall}. Indeed, here (as we have seen) we are going to keep the {\bf idle generators} together with the "working ones", while in \cite{Ovall} the idle generators (identified, roughly speaking, as the ones whose trace on the curved edge is too small) are just eliminated from the set of unknowns. Clearly this is done after a suitable change of basis has been performed in the space of trace generators. We acknowledge the fact that this makes the whole presentation, and partly the code as well, much simpler than our approach. On the other hand we think that the decision on "how lazy a trace generator has to be, in order to justify its elimination" could become delicate, and potentially giving rise to some instabilities in the borderline cases. All considered we believe that having both possibilities at hand could be  good for the scientific community.
\end{remark}

\subsection{Elements with one edge in $\Gamma_D$}

When $\PP$ is an element with one curved edge $\gamma$ in 
$\Gamma_D$ we proceed in a different way. For every function $\psi \in H^1(\gamma)$ (including the cases $\psi=0$, and $\psi=g_D$) we define ${\calV}_{k,\psi}(\PP)$ as follows.
\begin{equation}\label{defVkPG}
{\calV}_{k,\psi}(\PP):=\{v\in C^0(\overline{\PP}) \mbox{  such that }
v_{|\gamma}\!=\!\psi, 
~ v_{|e}\!\in\! \Po_k(e) 
\mbox{ on edges } e\!\neq\!\gamma, \mbox{ and } \Delta v\in\Po_{k-2}(\PP)
\} .
\end{equation}
Clearly, {\it once $\psi$ has been given}, in order to identify 
an element of ${\calV}_{k,\psi}(\PP)$ we must prescribe (in addition to the knowledge of $\psi$):
\begin{equation}
\begin{aligned}\label{dofGam}
&\mbox{- the values at each vertex of $\PP$ not on 
$\gamma$},\\
&\mbox{- (for $k\ge 2$) the values at the $k-1$ Gauss-Lobatto points of each
straight edge of $\PP$},\\
&\mbox{- (for $k\ge 2$) the moments of order $\le k-2$ inside $\PP$}.
\end{aligned}
\end{equation}
 
Note that, in this case, we could actually use the term {\it degrees of freedom}. Indeed, once $\psi$ has been fixed, the mapping from the above quantities to the elements of ${\calV}_{k,\psi}(\PP)$ is injective. Note also that for a general $\psi$ the affine manifold  
${\calV}_{k,\psi}(\PP)$ will fail to contain all polynomials of $\Po_k$, but whenever $\psi$ is the trace of a polynomial $p_k\in\Po_k$ then such a $p_k$ will belong to ${\calV}_{k,\psi}(\PP)$ (so that the patch-test will not be jeopardized).

For homogeneity of notation, here too we will consider the set of {\it generators} ${\bf G}\{{\calV}_{k,\psi}(\PP)\}$
as in \eqref{defbG}, although this, obviously, for $\psi\not\equiv 0$
will not be a linear space, but only an affine manifold.

\section{The local VEM stiffness matrices}\label{locstiff}

We define for $u$ and $v$ in $H^1(\Omega)$
\begin{equation}\label{splita}
a^\PP(u,v):=\int_{\PP}\pippof\nabla u \cdot \nabla v\,{\rm d}\Omega\quad \mbox{ and }a(u,v):=\sum_{\PP}a^\PP(u,v),
\end{equation}
and we point out that, obviously,
\begin{equation}\label{CS}
a^\PP(u,v)\le \pippof_{|\PP}\|u\|_{1,\PP}\,\|v\|_{1,\PP}, \quad \mbox{ and } \quad a(u,v)\le \pippof^{*}\|u\|_{1,\Omega}\,\|v\|_{1,\Omega}, \quad \mbox{ with }\pippof^{*}=\max\{\pippoi,\pippoz\},
\end{equation}
as well as (using Poincar\'e's inequality)
\begin{equation}\label{Poi}
 a(v,v)\ge C_{*} \pippof_{*}\|v\|_{1,\Omega}^2 \qquad \forall v \in H^1_{0,\Gamma_D}, \quad \mbox{ with }\pippof_{*}=\min\{\pippoi,\pippoz\}.
\end{equation}

 In this section we will construct for each element $\PP$ a bilinear form $a_h^\PP(\gu,\gv)$ (defined on generators of VEM spaces) to be used to approximate the {\it continuous} bilinear form $a^\PP$. Then, as done for Finite Elements,  we will define the {\it global} virtual element spaces, the {\it global} bilinear forms, and the {\it global} right-hand sides summing the contributions of the single elements.

\subsection{The $\Pi^{\nabla}_k$ projection operator}

Our first, fundamental, item will be (as common for Virtual Elements) the construction of the 
$\Pi^{\nabla}_k$ projection operator. Given an element $\PP$ and a function $v$ in $H^1(\PP)$, 
we construct  a polynomial $\Pi^{\nabla}_k v$  in $\Po_k$  defined by

\begin{equation}
\left\{
\begin{aligned}\label{defpinablak}
&\int_{\PP}\nabla \Pi^{\nabla}_k v\cdot\nabla q_k\,{\rm d}\PP=\int_{\PP}\nabla v \cdot \nabla q_k\,{\rm d}\PP\qquad\forall
q_k\in\Po_k, \\
&\int_{\partial{\PP}}\Pi^{\nabla}_k v\,{\rm d}s=\int_{\partial{\PP}}v\,{\rm d}s.
\end{aligned}
\right.
\end{equation}
We point out that for a $v\in \calV_k(\PP)$ {\color{black}(or in $\calV_{k,g}(\PP)$ for some $g\in H^1(\Gamma_D)$}) all the terms appearing in  \eqref{defpinablak} are actually computable (for all types of elements \eqref{def-Vk},
\eqref{def-calVk-c}, and \eqref{defVkPG})  from the knowledge  of the {\it degrees of freedom} (or of the {\it generators}) \eqref{dof-Vk}, or \eqref{dof}, or
\eqref{dofGam}, respectively. Indeed, we first note that both left-hand sides of
  \eqref{defpinablak} are integrals of polynomials over $\PP$ or 
$\partial\PP$, respectively . The right-hand side of the second equation in \eqref{defpinablak} is also computable, since $v$ is known on $\partial\PP$ (being either a polynomial, or the trace of a polynomial, or equal to $g$).
Finally, the right-hand side of the first equation is
\begin{equation}\label{byparts1}
\int_{\PP}\nabla v\cdot\nabla q_k{\rm d}\PP=-\int_{\PP}v \Delta q_k{\rm d}\PP+\int_{\partial\PP} v\, (\nabla q_k\cdot{\bf n}){\rm d}s.
\end{equation}
The first term in the right-hand side of \eqref{byparts1} is made of  moments of $v$
of order $k-2$ on $\PP$ (and hence computable from the "degrees of freedom" of $v$), and the second term is also computable since 
$v$ is either a polynomial, or a trace of a polynomial, or equal to $g$ on $\partial\PP$.

Once we defined the projection $\Pi^{\nabla}_k$, and checked that it is {\it computable}, in practice,
for every $v\in \calV_k(\PP)$,
we can extend it, in the obvious way, to an element $\gv$ in ${\bf G}\{\calV_k(\PP)\}$ by setting
\begin{equation}\label{pinagv}
\Pi_k^{\nabla}\gv:=\Pi_k^{\nabla}\fv.
\end{equation}
 Then we can follow the usual track of Virtual Elements, setting, for  $\gu$ and $\gv$ in ${\bf G}\{\calV_k(\PP)\}$: 
\begin{equation}\label{defahP}
a_h^{\PP}(\gu,\gv):=\int_{\PP}\pippof\nabla \Pi^{\nabla}_k\gu\cdot\nabla\Pi^{\nabla}_k\gv\,{\rm d}\PP + \S^{\PP}((\I-{\mathcal G}\Pi^{\nabla}_k)\gu,(\I-{\mathcal G}\Pi^{\nabla}_k)\gv)
\end{equation}
where $\I$ is the {\it identity} operator, $\mathcal G$ has been defined in \eqref{gen-pol}, and, as usual in VEM formulations, $\S^{\PP}$ is a symmetric positive semi-definite bi-linear form 
such that there exists a positive constant $\alpha_*$, independent of $h$, with
\begin{equation}\label{stabilityl}
\alpha_*a^{\PP}(\fv,\fv)\le a_h^{\PP}(\gv,\gv)\quad\forall\, \gv\in {\bf G}\{\calV_k(\PP)\}.
\end{equation}

\noindent We immediately point out that, independently of the choice of 
$\S^\PP$, we will always have the consistency property
\begin{equation}\label{consistency}
a^\PP_h({\mathcal G} p_k,\gv)=a^\PP({ p}_k,\Pi^{\nabla}_k\fv)\equiv a^\PP({ p}_k,\fv)\quad\forall  p_k\in\Po_k,\;\forall \gv\in {\bf G}\{\calV_k(\PP)\}. 
\end{equation}
Coming now to the choice of $\S^\PP$, we note that there are {\it many} guidelines available in the VEM literature in order to construct a stabilizing bilinear 
form $\S^{\PP}$. 
Here, however, we {\it must} distinguish between bilinear forms that one could apply
to generic {\it elements $v$ of the VEM space} $\calV_k(\PP)$,
and bilinear forms that can also be applied to the generators $\gv\in {\bf G}\{\calV_k(\PP)\}$. An example of the first case is
 the rather common 
\begin{equation}\label{Sintebor}
\S^{\PP}(u,v):=h_{\PP}^{-1} {\pippof_{|\PP}}\int_{\partial\PP} u\,v{\rm d}\ell
\end{equation}
where $h_{\PP}$ is, say, the diameter of $\PP$. An example of the second case is the (equally classic)
\begin{equation}\label{dofi-dofi}
\S^{\PP}(\gu,\gv):={\pippof_{|\PP}}\sum_i \delta_i(\gu)\delta_i(\gv),
\end{equation}
where
\begin{itemize}
\item each $\delta_i(\gv)$  is the $i$-$th$ term of the generator $\gv$,
\item the sum is extended to all of them.
 \end{itemize}
Note, however, that {\it  here} a choice like \eqref{Sintebor} {\bf cannot}
be used: indeeed, with the choice \eqref{Sintebor} a $\gv$ having $\fv\equiv 0$ will not be stabilized, and \eqref{stabilityl}
will not hold.
%
%
Hence, to fix the ideas, we will assume here that the stabilizing term $\S^{\PP}(\gu,\gv)$ has exactly the form
\eqref{dofi-dofi}. In this case, under suitable assumptions on the mesh, that will be discussed in Sect. \ref{interr}, \eqref{stabilityl} will always hold.
At the same time, we also have
\begin{equation}\label{magdelta}
a_h^\PP(\gv,\gv)+ \|\Pi^{0,\PP}_{0}\fv\|^2_{0,\PP}\ge \sigma_1\, 
\sum_i (\delta_i(\gv))^2
\end{equation}
for some other constant $\sigma_1>0$ independent of $h$.

\begin{remark}\label{daunaparte} Actually, as we shall see with much more detail in what follows, we do not need  \eqref{stabilityl}
to hold separately in each element. We just need that {\bf the sum} (over all the elements $\PP$ in $\Th$) satisfies the analogue of  \eqref{stabilityl}. Indeed, thinking about that, it is clear that the degrees of freedom common to two
or more elements would be  "stabilized more than once" by applying the stabilization 
 in each element. Hence we might consider the possibility of stabilizing {\bf each} of them {\bf in one element} only. In most cases taking this into account would only allow a minor saving in the amount of operations, paid, on the other hand, with some additional complications in the code. In some other cases, however, (typically when a curved edge separates two elements where $\pippof$ assumes very different values) stabilizing only {\bf once} would be the best choice in order to preserve accuracy. We will come back to that in Sec. \ref{sec:Details}. \hfill\qed 
\end{remark}
\begin{remark}\label{stabprob}
We recall that the classical stability condition for VEM stiffness matrices would  be to require the existence of {\it two} positive constants   
$\alpha_*$ and  $\alpha^*$ such that
\begin{equation}\label{stability}
\alpha_*a^{\PP}(v,v)\le a_h^{\PP}(v,v)\le\alpha^* a^{\PP}(v,v)\quad\forall\, v\in \calV_k(\PP).
\end{equation}
This with the present notation would be
\begin{equation}\label{stabilityG}
\alpha_*a^{\PP}(\fv,\fv)\le a_h^{\PP}(\gv,\gv)\le\alpha^* a^{\PP}(\fv,\fv)\quad\forall\, \gv\in {\bf G}
\{\calV_k(\PP)\}.
\end{equation}
However, there are cases where the second inequality in \eqref{stabilityG} would be  very difficult (if not definitely impossible) to obtain. This could happen 
in the presence of {\bf idle generators}. For instance, when $\gamma$ is straight there will  be generators $\gv$ (corresponding to trace generator points "attached to $\gamma$ but not belonging to $\gamma$") such that $\fv$ is identically zero on $\gamma$ (and then in the whole $\PP$), producing in \eqref{stabilityG} a left-hand side and a right-hand side that are both equal to zero. If these dof's were not stabilized, the final stiffness matrix would end up being singular, but stabilizing them (for instance with the  choice \eqref{dofi-dofi}) would make the right inequality in \eqref{stabilityG} impossible.
\hfill\qed
\end{remark}

\section{The global spaces and the global $a_h$}\label{glostiff}

\subsection{Approximations of $H^1_{0,\Gamma_D}$
and $H^1_{g_D,\Gamma_D}$}

To start with, we shall now design the global space and the global {\it affine manifold} to be used as approximations  of $H^1_{0,\Gamma_D}$
and $H^1_{g_D,\Gamma_D}$ (respectively) in order to discretize 
\eqref{proconvar}.

For this, we define first, for every function $\psi \in H^1(\Gamma_D)$:
\begin{multline}\label{vem:definition}
\calV_{k,\psi}(\Omega):=\{v\in C^0(\overline{\Omega}) \mbox{ such that } v_{|\PP}\in 
\calV_{k}(\PP)\;\, \forall \PP\in \Th \mbox{ without edges in $\Gamma_D$}\\
\mbox{ and }
v_{|\PP}\in 
\calV_{k,\psi}(\PP)\;\, \forall \PP\in \Th \mbox{ with an edge in } \Gamma_D
\},
\end{multline}
{\color{black}and we observe that, for $\psi=0$, $\calV_{k,0}(\Omega)$ is a linear space.}
The {\it generators} for  $\calV_{k,\psi}(\Omega)$  will be:
\begin{itemize}
\item The values at the vertices in $\Omega\cup\Gamma_R$. One unknown per each such vertex. Remember that
$\Gamma_R$ is an {\it open} subset of $\partial\Omega$, so that the points belonging to the closure of  ${\Gamma_D}$ are excluded. 
\item (for $k\ge 2$) The values at the $k-1$ Gauss-Lobatto points of each (straight) edge internal to $\Omega$. Hence $k-1$ unknowns per internal edge. 
\item The values at the {\it trace generator points} of each curved edge,
on $\partial\Omega_i$ and on $\Gamma_R$. These are $\pi_{k}-2$ for each such edge.
\item (for $k\ge 2$) The moments of order $\le k-2$ internal to each element 
($\pi_{k-2}$ unknowns per each element).
\end{itemize}

Here too we have that a generator $\gv$, together with a function $\psi$ in $H^1(\Gamma_D)$ will identify uniquely an element in $\calV_{k,\psi}(\PP)$, but the converse will not be true.

As common in the Finite Element codes,  the affine manifold $\calV_{k,g_D}$ could be constructed as 
\begin{equation}
\calV_{k,g_D}(\Omega):=\ghD+
\calV_{k,0}(\Omega)
\end{equation}
where $\overline{g_D}$ is a single fixed element of  
$\calV_{k,g_D}$ suitably constructed. 
 Typically, one takes
$\overline{g_D}$ to be {\it the} element in $\calV_{k,g_D}$
that in each $\PP$ has all the degrees of freedom \eqref{dofGam} equal to zero.

Here, for simplicity, we will treat {\it every} edge $e$ in 
$\Gamma_D$ as if it was {\it a curved edge}. Indeed, as we already pointed out several times, we want our approach to accept straight boundary edges as  particular cases of the curved ones, in particular since we do not want to enter the details of edges that are {\it only slightly curved}. 
Needless to say, if a whole part of $\Gamma_D$
(or even the  whole $\Gamma_D$) is known to be straight, then the corresponding elements will be treated as {\it normal polygons}.

\begin{remark}\label{trattog}
The present approach for dealing with Dirichlet boundary conditions is surely the one that is simpler to present, and at the same time simpler to analyze (as we shall see in a while).
However, depending on the way in which the code has been structured, it might create some difficulties in the implementation, and in particular its extension to the cases of elements with more than one curved boundary edge could become
complicated to deal with in the computer code.  Actually, here 
too we might proceed as is also common in Finite Element codes, constructing first the local stiffness matrix as for an internal element, and then rearrange rows and columns (and right-hand side) of the final global stiffness matrix in order to take into account that some boundary unknowns are actually  known, and given by $g$. 
\end{remark}

\subsection{Approximation of the bilinear forms and right-hand sides}

We are now ready to  write the {\it global} quantities. {\color{black}For any function $\psi\in H^1(\Gamma_D)$, and for every $\gu$ and $\gv$ in
${\bf G}\{\calV_{k,\psi}(\Omega)\}$, we set}
\begin{equation}\label{splitah}
a_h(\gu,\gv)=\sum_{\PP\in\Th}a_h^{\PP}(\gu,\gv)
\end{equation}
where $a_h^{\PP}$ has been defined in \eqref{defahP}. The boundary integrals 
\begin{equation}
\int_{\Gamma_R}\Ro\, \fu\, \fv\,{\rm d}\ell, \quad \mbox{and} \quad \int_{\Gamma_R}g_R \fv\,{\rm d}\ell
\end{equation}
are not a difficulty for $\fu$ and $\fv$ in $\calV_{k,\psi}(\Omega)$. It is clear that from \eqref{stabilityl} we immediately have
\begin{equation}\label{stabilitylg}
\alpha_*a(\fv,\fv)\le a_h(\gv,\gv)\quad\forall\, \gv 
\in {\bf G}\{\calV_{k,0}(\Omega)\}.
\end{equation}
In order to have some upper bound (similar to the one in \eqref{stabilityG}) to be used to prove error estimates, we define a new norm 
\begin{equation}\label{triplebarnorm}
\| \gv\|_{1,S}:=(a_h(\gv,\gv))^{1/2}.
\end{equation}
This allows us to re-write the stability \eqref{stabilitylg} as
\begin{equation}\label{stabilitylgf}
\alpha_*a(\fv,\fv)\le a_h(\gv,\gv)\equiv\| \gv\|_{1,S}^2\qquad\forall\, \gv \in{\bf G}\{\calV_{k,0}(\Omega)\}.
\end{equation}
%
We also note that, naturally, 
\begin{equation}\label{CSh}
a_h(\gu,\gv)\le\|\gu\|_{1,S}\,\|\gv\|_{1,S}\;\qquad\forall\, \gu,\gv \in{\bf G}\{\calV_{k,0}(\Omega)\}.
\end{equation}
Summing \eqref{magdelta} over the elements, we have
\begin{equation}\label{magdeltOom}
a_h(\gv,\gv)+ \sum_{\PP\in \Th}\|\Pi^{0,\PP}_0\fv\|^2_{0,\PP}\ge\sigma_1\, 
\sum_i (\delta_i(\gv))^2 \qquad\forall\, \gv \in{\bf G}\{\calV_{k,0}(\Omega)\}
\end{equation}
where now the sum is extended to all indices in
${\bf G}\{\calV_{k,{0}}(\Omega)\}$.  In turn, using the properties of projection operators and Poincar\'e
inequality one has
\begin{equation}\label{Poinca}
\sum_{\PP\in \Th}\|\Pi^{0,\PP}_0\fv\|^2_{0,\PP}\le
\|\fv\|^2_{0,\Omega}\le C\,a(\fv,\fv) \qquad\forall\, \gv \in{\bf G}\{\calV_{k,0}(\Omega)\}
\end{equation}
for some constant $C$ independent of $h$.  The combined use of \eqref{magdeltOom}, \eqref{Poinca}, and \eqref{stabilitylgf} 
gives then
\begin{equation}\label{ahellipt}
a_h(\gv,\gv)\ge\sigma_2\, 
\sum_i (\delta_i(\gv))^2 \qquad\forall\, \gv \in{\bf G}\{\calV_{k,0}(\Omega)\}
\end{equation}
for some constant $\sigma_2$ independent of $h$.

Next, on every element $\PP$, let $\fv$ be the unique function associated with the set of generators $\gv$. We define
\begin{equation}\label{per-fh}
\T_{\PP}\fv:=\begin{cases}
\displaystyle\Pi^{\nabla}_1 \fv~~&\mbox{for }k=1\\[2mm]
\displaystyle\Pi^{0,\PP}_{k-2} \fv~~&\mbox{for }k\ge 2
\end{cases}
\end{equation}
and 
\begin{equation}
\T \gv\equiv \T \fv:=
\{\T_{\PP}\fv \mbox{ in every }\PP\in\Th\} .
\end{equation}
Then we set 
\begin{equation}\label{effeh}
<f_h,\gv>:=\int_{\Omega}
f\,\T \fv\,\dO \equiv (f,\T\fv)_0,
\end{equation}
and finally,
\begin{equation}
<g_R,\gv>:=\int_{\Gamma_R}g_R\,\fv\,{\rm d}\ell. 
\end{equation}


%

\section{The discretized problem}\label{dispro}


We can now write the discretized version of problem \eqref{proconvar}:

\begin{equation}
\left\{
\begin{aligned}\label{proconvarh}
& \mbox{find } \gu_h\in {\bf G}\{\calV_{k,(g_D,\Gamma_D)}\} \mbox{ such that }\\
& a_h(\gu_h,\gv)
+\int_{\Gamma_R}\Ro\,\fuh\,\fv\,{\rm d}s=<f_h,\gv> + <g_R,\gv>\;\;
\forall \gv\in {\bf G}\{\calV_{k,(0,\Gamma_D)}\}.
\end{aligned}
\right.
\end{equation}

We point out that whenever the exact solution $u$ is in $\Po_k(\Omega)$, the solution $\gu_h$ of
\eqref{proconvarh} will satisfy $\fuh\equiv u$, so that the {\it patch-test of order k} will hold true.

\subsection{Error estimates}
We have the following error estimates.
\begin{thm}\label{th-stimerr}
In the above assumptions, problem \eqref{proconvarh} has a unique 
solution $\gu_h$. Moreover, there exists a constant $C$, depending only on 
the value of $\alpha_*$ in \eqref{stabilitylgf} and on 
the data $\Ro$, ${\pippoz}$  and  ${\pippoi}$,  such that:
if  $u$ is the solution of \eqref{proconvar}, then for every $\gu_I$ in ${\bf G}\{\calV_{k,(g_D,\Gamma_D)}\}$ and for every 
 $u_{\pi}$ elementwise in $\Po_k(\PP)$ we have 
\begin{equation}\label{ErrEst}
\|u-\fuh\|_1\le C\Big(\| \gu_{\pi}-\gu_I\|_{1,S}+\|u-u_{\pi}\|_{1,h}+\|u-{\fuI}\|_1+
{\mathcal E}_1(f) \Big)
\end{equation}
where $\gu_{\pi}={\mathcal G}u_{\pi}$  (see \eqref{gen-pol}),
$\|\cdot\|_{1,h}$ is the "broken norm" in $\prod_{\PP} H^1(\PP)$ and ${\mathcal E}_1(f)$ 
is given by
\begin{equation}\label{defEf}
{\mathcal E}_1(f):=
\sup_{\gv\in {\bf G}\{\calV_{k,(0,\Gamma_D)}\}}\frac{(f,\fv)-<f_h,\gv>}{\|\fv\|_1}\equiv
\sup_{\fv\in\calV_{k,(0,\Gamma_D)}}
\frac{(f,\fv-T\fv)_0}{\| \fv\|_1} .
\end{equation}
\end{thm}


\begin{proof}
Uniqueness is an immediate consequence of \eqref{ahellipt}.
Let then $\gu_h$ be the solution of \eqref{proconvarh} and let $\gu_I$ be an element of 
${\bf G}\{\calV_{k,(g_D,\Gamma_D)}(\Omega)\}$. 
We set 
\begin{equation}\label{defdeltah}
\bdelta:=\gu_h-\gu_I, 
\end{equation} 
and we remind that
\begin{equation}\label{deltazero}
\bdelta\in {\bf G}\{\calV_{k,(0,\Gamma_D)}\}  \mbox{ while }\quad \bdeltah\in H^1_{0,\Gamma_D},
\end{equation} 
implying in particular, from \eqref{Poi} and \eqref{stabilitylgf},
\begin{equation}\label{confronto-norme}
\|\bdeltah\|^2_1\le C\,a(\bdeltah,\bdeltah)\le C\|\bdelta\|^2_{1,S}.
\end{equation}
To simplify the notation (and make it similar to more classical ones) we also set
\begin{equation}
u_h:=\fuh,\qquad u_I:=\fuI,\qquad \deltas:=\bdeltah\mbox{ and }\quad[v,w]_{\Gamma_R}:=\int_{\Gamma_R}\rho\,v\,w\,{\rm d}s .
\end{equation}
We point out in advance that, since $u_{\pi}$ is piecewise in $\Po_k$, from \eqref{consistency} we have
\begin{equation}\label{split}
a_h(\gu_{\pi},\bdelta)=\sum_{\E}a^{\E}(u_{\pi},\deltas),
\end{equation}
and moreover that (using classical trace theorems and then \eqref{confronto-norme}) there exists a constant 
$C$, independent of $h$, such that for all $\gv\in {\bf G}\{\calV_{k,(0,\Gamma_D)}\}$ 
\begin{equation}\label{trace}
\|\fv\|_{0,\Gamma_R}\le C \|\fv\|_1\le C\|\gv\|_{1,S}.
\end{equation}
Then we proceed using (successively): \eqref{triplebarnorm},\eqref{defdeltah}; rearranging terms; \eqref{proconvarh} and $\pm \gu_{\pi}$;  \eqref{split}; \eqref{effeh},   $\pm u$ (twice) and \eqref{splita}; rearranging;   
\eqref{proconvar}; \eqref{defEf}, \eqref{CSh}, \eqref{CS}; finally \eqref{confronto-norme}  and \eqref{trace}. We get:
\begin{equation}\label{stimerr}
\begin{aligned}
&\|\bdelta\|^2_{1,S}=a_h(\bdelta,\bdelta)\le a_h(\bdelta,\bdelta)\!
+[\deltas,\deltas]_{\Gamma_R}=\!a_h(\gu_h,\bdelta)-a_h(\gu_I,\bdelta)+[u_h,\deltas]_{\Gamma_R}-[u_I,\deltas]_{\Gamma_R}\\[2mm] 
&=\!a_h(\gu_h,\bdelta)+[u_h,\deltas]_{\Gamma_R}-a_h(\gu_I,\bdelta)-[u_I,\deltas]_{\Gamma_R}\\[2mm]
&={\color{black}<f_h,\bdelta>}+\!<g_R,\deltas>\!-\! a_h(\gu_I-\gu_{\pi},\bdelta)-a_h(\gu_{\pi},\bdelta)-[u_I,\deltas]_{\Gamma_R}\\[2mm]
&=<f_h,{\bdelta}>+\!<g_R,\deltas>\!-\!a_h(\gu_I-\gu_{\pi},\bdelta)-\sum_{\E}a^\E(\up,\deltas)
-[u_I,\deltas]_{\Gamma_R}
\\
&=(f,\T\deltas)_0+\!<g_R,\deltas>\!-\! a_h(\gu_I-\gu_{\pi},\bdelta)-\sum_{\E}a^\E(\up-u,\deltas)\\
&\hskip2cm-a(u,\deltas)
-[u_I-u,\deltas]_{\Gamma_R}-[u,\deltas]_{\Gamma_R}
\\
&=(f,\T\deltas)_0+\!<g_R,\deltas>\!-a(u,\deltas)-[u,\deltas]_{\Gamma_R}-\! a_h(\gu_I-\gu_{\pi},\bdelta)-\sum_{\E}a^\E(\up-u,\deltas)\\
&\hskip2cm
-[u_I-u,\deltas]_{\Gamma_R}
\\
&=(f,\T\deltas-\deltas)_0\!\!-a_h(\gu_I-\gu_{\pi},\bdelta)-\sum_{\E}a^\E(\up-u,\deltas)
-[u_I-u,\deltas]_{\Gamma_R}
\\
&\le \,C\Big({\mathcal E}_1(f)\|\theta\|_1+\| \gu_I-\gu_{\pi} \|_{1,S}\|\bdelta\|_{1,S}+\|u-\up\|_{1,h}\|\theta\|_1+\|u_I-u\|_{0,\Gamma_R}\|\theta\|_{0,\Gamma_R}\Big)\\
&\le\,C\Big({\mathcal E}_1(f)
+\| \gu_I-\gu_{\pi} \|_{1,S}
+\|u-\up\|_{1,h}+\|u_I-u\|_{1}\Big)\,\|\bdelta\|_{1,S} .
\end{aligned}
\end{equation}
Thus, 
\begin{equation}
\|\gu_h-\gu_I\|_{1,S}\le\,C\Big({\mathcal E}_1(f)
+\| \gu_I-\gu_{\pi} \|_{1,S}
+\|u-\up\|_{1,h}+\|u_I-u\|_{1}\Big) .
\end{equation}
The result follows by the triangle inequality.
\end{proof}

\bigskip


\section{Interpolation estimates}\label{interr}

Looking at Theorem \ref{th-stimerr} we see that in order to get the final error estimate in the usual terms of {\it powers of $h$ and regularity of the solution} we need to choose $\gu_I$ and $\gu_\pi$ and then estimate (for them) the quantities
\begin{equation}\label{dastimare}
\|u-u_{\pi}\|_{1,h},\qquad \|u-u_I\|_{1}\qquad\|{\gu_{\pi}-\gu_I}\|_{1,S},\qquad\mbox{ and } {\mathcal E}_1(f).
\end{equation}
As we shall see, the first two terms and the last one in   \eqref{dastimare} could be treated in a reasonably standard way. The third, however, poses some problems, as the norm $\|\cdot\|_{1,S}$ takes into account the choice of the stabilization, and, as we shall see, care has to be taken in order to choose a $\gu_I$
and a $\gu_{\pi}$ that make the third term small without affecting the accuracy of the other two.


We make the following mesh assumption.

\begin{itemize}
\item[{\bf A)}] The boundary of each element $\calP$ is piecewise $C^1$. Moreover we assume the existence of a positive constant $\newthetas$  
such that all elements $\calP$ of the mesh are star-shaped with respect to a ball $B_\calP$ of radius 
$\rho_\calP 
\ge \newthetas h_\calP$ (where, as usual, $h_\calP$ is the diameter of $\calP$). 
\item[{\bf B)}] There exists a positive constant $\newthetas_1$ such that the length of each edge of $\PP$ is $\ge \newthetas_1 h_{\PP}$.
\end{itemize}

\subsection{More details on the stabilization}\label{sec:Details}

Before going to the study of the interpolation errors we must spend some more attention to the choice of the stabilization term (that enters in the definition of the $\|\cdot\|_{1,S}$ norm).  For simplicity we will only deal with elements with a curved edge, since the case of polygonal elements has already been treated in {\color{black}
\cite{BLR-stab}, \cite{Brenner}}.

Moreover, always for simplicity,  we fix our attention on the {\it dofi-dofi} stabilization \eqref{dofi-dofi}. Under our assumptions on the mesh the validity of \eqref{stabilityl}, that is, estimates from below,  can be easily proved with the techniques used in \cite{BLR-stab}, \cite{Brenner}. Therefore we only have to deal with estimates from above.

We already pointed out in Remark \ref{daunaparte} that the degrees of freedom shared by two (or more) elements do not need to be stabilized independently in each element, but, in the {\it dofi-dofi} stabilization,
 they could be considered only {\bf once} (in either one of the two elements). 

We also pointed out, in Remark \ref{stabprob}, that troubles may arise only when considering the tgp associated with curved edges that separate two elements having a (quite) different material coefficient $\pippof$. Indeed, as we shall see in a little while, when a curved edge separates two elements (say, $\PP$ and $\PP^\prime$) such that the exact solution 
$u$ is globally regular on $\overline{\PP}\cup\overline{\PP^\prime}$ the estimate can be performed following classical arguments. However, in the presence of a jump in the coefficient $\pippof$, the exact solution will have a jump in the normal derivative at the interface,
so that its global regularity on ${\PP}\cup\gamma\cup{\PP^\prime}$  cannot be better than $H^{3/2-\varepsilon}$ with $\varepsilon > 0$.
Indeed, as the tgp attached to the curved interface $\gamma$, seen from $\PP$ and from $\PP^{\prime}$, are {\it the same}, then the value of $u_I$ (to be chosen) in a tgp cannot be close, at the same time, to the value of $u$ in one element and to a smooth extension of the values that 
$u$ assumes in the other element, and this is the reason why we {\bf must} stabilize the tgp of the interface $\gamma$ only {\bf once}.

Hence for every curved edge $\gamma$ that separates two elements $\PP$ and   $\PP^{\prime}$ having a different value of 
$\pippof$ we choose, once and for all, one of the two elements (say, $\PP$) and we consider the values of the 
tgp, in the {\it dofi-dofi} stabilization, only when dealing with the element $\PP$. In the description below, we will assume, {\it to simplify  the exposition}, that the element $\PP$ chosen for the stabilization is the one that contains "more tg points", and that when considering the values of $u$ at tg points not belonging to $\PP$ we will actually consider a smooth extension of $u_{|\PP}$.
Note that this extension has to be done just when proving estimates, and (obviously) {\bf not} in the code.

\subsection{Construction of $u_I$ and $u_{\pi}$ - Standard  case}

For the sake of clarity, we will first discuss the simpler case where either the curved edge $\gamma$ belongs to
$\partial \Omega$ or $\gamma$ is internal to
$\Omega$ but the material coefficient $\pippof$ varies smoothly from one element to the other. As we have said already, in such cases we don't have to worry about the effects of the stabilizing terms, and we might assume that
(for simplicity) the stabilization is done, independently, in each of the two elements.

In this case (that we denoted as {\it Standard Case}) we proceed as follows. We set: 

$$\widetilde{\calP}:=\calP\cup\gamma\cup\calP^{\prime}$$
when the curved edge $\gamma$ is shared by two elements, and
$$\widetilde{\calP}:=\calP$$
otherwise. We note that, always in the standard case, the exact solution $u$ will be smooth in $\widetilde{\calP}$. Then we take a polynomial $q_k$ living on $\widetilde{\calP}$ such that $q_k(\nu)=u(\nu)$ at the two endpoint of $\gamma$ and 
\begin{equation}\label{stimau-qk}
\| u - q_k \|_{r,\widetilde{\calP}} \le C h_{\widetilde{\calP}}^{t-r} |u|_{t,\calP} \;\mbox{for all real numbers $0 \le r \le t \le k+1$ and $t>1$}.
\end{equation}
 Due to assumption {\bf A} it can be checked that such polynomial exists. Then we take, in both elements 
$\calP$ and $\calP^{\prime}$ 
\begin{equation}
u_\pi\equiv q_k ,
\end{equation} 
and this, using  \eqref{gen-pol}, defines $\gu_{\pi}$ as ${\mathcal G}u_{\pi}$.

The  local interpolant $u_I$, defined through its generators $\gu_I$, will be constructed separately in $\calP$ and $\calP^{\prime}$. We will describe its construction in $\calP$, and estimate $u-u_I$ in $\calP$. The construction and the estimate in $\calP^\prime$ are done exactly in the same way.  

We take $u_I (\nu) = u(\nu)$ for all points $\nu$ that correspond to a vertex of $\calP$ or to a Gauss-Lobatto node of a straight edge of $\calP$, and
%
\begin{equation}\label{uIgamma}
{u_I}_{|\gamma} = {q_k}_{|\gamma}.
\end{equation} 
Inside $\calP$ we require $\Delta u_I = \Pi^0_{k-2}(\Delta u)$.  This defines a $u_I$ as an element of ${\mathcal V}_k(\PP)$. In view of Remark \ref{Gpol}, in order to define $\gu_I$ we only need to prescribe its values at the tgp:
\begin{equation}\label{uInodi}
\gu_I (\nu) = q_k(\nu) \quad \mbox{at all the tgp of the curved edge $\gamma$}.
\end{equation}
 Now we have to prove the interpolation estimate. In the following $C$ will denote as usual a generic constant, uniform in the mesh size and shape, that can possibly change at each occurrence. In what follows all the Sobolev norms on $\partial\calP$ are, as usual, defined with respect to the arclength parametrization.

Before delving into the interpolation proofs, we need to introduce some definition and two simple technical lemmas. 
Let the scaled norms (for all $0 \le \eps < 1/2$)
$$
\begin{aligned}
& \tri v \tri_{1/2+\eps,\partial \calP} = h_{\PP}^{-1/2-\eps} \| v \|_{0,\partial \calP} + | v |_{1/2+\eps,\partial \calP} \\
& \tri v \tri_{1+\eps,\calP} = h_{\PP}^{-1-\eps} \| v \|_{0,\calP} + | v |_{1+\eps,\calP} \: .
\end{aligned}
$$
Under the hypothesis {\bf A}, following Lemma 3.3 in \cite{rozzi}, for every element $\calP$ one can build a $W^{1,\infty}$ mapping ${\bf F}_\calP$ from $B_\calP$ into $\calP$, with inverse ${\bf F}_\calP^{-1}$ also in $W^{1,\infty}$, in such a way that the $W^{1,\infty}$ norm of ${\bf F}_\calP$ and ${\bf F}_\calP^{-1}$ are uniformly bounded independently of $\calP$. 
{\color{black}
\begin{lem}\label{lemma8.1}
Let assumption {\bf A} hold. Then there exists a uniform constant C such that (for all elements $\calP$)
\begin{equation}\label{L:prel-2}
\begin{aligned}
& | v |_{1,\calP} \le C |v|_{1/2,\partial \calP} \qquad \forall v \in H^{1}(\calP) \textrm{ with } \Delta v = 0 \: , \\
& \tri v \tri_{1,\calP} \le C \tri v \tri_{1/2,\partial \calP} \qquad \forall v \in H^{1}(\calP)
\textrm{ with } \Delta v = 0 \: .
\end{aligned}
\end{equation}
\end{lem}
\begin{proof}
Given any $v \in H^1(\calP)$ with $\Delta v =0$, we denote by $\tilde v$ the only function in $H^1(B_\calP)$ with $\Delta \tilde v = 0$ and such that $\tilde v = v \circ {\bf F}_\calP$ on $\partial B_\calP$. We recall that the map ${\bf F}_\calP$ and its inverse are uniformly in $W^{1,\infty}$.
First using that $v$ is harmonic in $\calP$ (and thus minimizes the $H^1$ semi-norm among all $H^1$ functions with same boundary values), then recalling that $\tilde v$ is harmonic in $B_\calP$, 
the first bound in \eqref{L:prel-2} follows
\begin{equation}\label{newnew3}
| v |_{1,\calP} \le | \tilde v \circ {\bf F}_\calP^{-1}  |_{1,\calP} \le C | \tilde v  |_{1,B_\calP}
\le C | \tilde v  |_{1/2,\partial B_\calP} = C | v \circ {\bf F}_\calP |_{1/2,\partial B_\calP}
\le C |v|_{1/2,\partial \calP} . 
\end{equation}
It is well known, see for instance \cite{Necas}, that for all Lipschitz domains $\omega$ there exists a constant $C$ such that 
$\| w \|_{0,\omega} \le C (|w|_{1,\omega} + \| w \|_{0,\partial\omega})$ for all $w$ in $H^1(\omega)$. By a standard scaling argument with the unitary ball, the previous result immediately yields (for any element $\calP$)
\begin{equation}\label{newnew1}
\rho_\calP^{-1} \| w \|_{0,B_\calP} \le C( |w|_{1,B_\calP} + \rho_\calP^{-1/2} \| w \|_{0,\partial B_\calP})
\quad \forall w \in H^1(B_\calP)
\end{equation}
with $C$ independent of $\calP$. By mapping to the ball, applying result \eqref{newnew1} and mapping back to $\calP$ (we also recall that $\rho_{\calP} \sim h_\calP$) we obtain
\begin{equation}\label{newnew2}
\begin{aligned}
h_\calP^{-1} \| v \|_{0,\calP} 
& \le  \rho_\calP^{-1} \| v \circ {\bf F}_\calP^{-1} \|_{0,B_\calP}  
\le  C (|v \circ {\bf F}_\calP^{-1}|_{1,B_\calP} + \rho_\calP^{-1/2} \| v \circ {\bf F}_\calP^{-1} \|_{0,\partial B_\calP}) \\
& \le C( |v|_{1,\calP} + h_\calP^{-1/2} \| v \|_{0,\partial \calP}) .
\end{aligned}
\end{equation}
The second bound in \eqref{L:prel-2} follows from the definition of the $\tri\cdot\tri$ norms, \eqref{newnew2} and \eqref{newnew3}.
\end{proof}
}

\begin{lem}
Let assumption {\bf A} hold. Then there exists a uniform constant C such that (for all elements $\calP$)
\begin{equation}\label{newnew5}
\| v \|_{0,\calP} \le C \Big(h_\calP | v |_{1,\calP} +\Big|\int_{\partial \PP} v \Big| \Big)
\qquad \forall v \in H^1(\calP) .
\end{equation}
\end{lem}
\begin{proof}
We observe first that, under our mesh assumptions  it holds
\begin{equation}\label{constant}
\|\varphi\|_{0,\PP}\le C \Big| \int_{\partial \PP} \varphi \Big| \quad \forall \varphi \mbox{ constant}.
\end{equation}
Indeed:
\begin{equation*}
\|\varphi\|^2_{0,\PP}=|\PP| \varphi^2 =\frac{|P|}{|\partial \PP|^2}\Big(\int_{\partial \PP}\varphi\Big)^2\le C \Big(\int_{\partial \PP}\varphi\Big)^2 ,
\end{equation*}
with $C$ independent of $h_{\PP}$.
For $v\in H^1(\PP)$, let $\overline v$ be its average on $\calP$. By standard approximation properties  and \eqref{constant} we have
\begin{equation}\label{uno}
\|v\|_{0,\PP}\le \|v -\overline v\|_{0,\PP} +\|\overline v\|_{0,\PP}\le C\Big( h_{\PP}|v|_{1,\PP}+\Big| \int_{\partial \PP} \overline v\Big|).
\end{equation}
Next, by adding and subtracting $v$ and the usual argument mapping ``to and from'' the ball $B_\calP$ combined with a scaled trace inequality on the ball, we can easily derive
\begin{equation}\label{due}
\begin{aligned}
\Big| \int_{\partial \PP} \overline v\Big|&\le \Big| \int_{\partial \PP} (\overline v-v)\Big| +\Big| \int_{\partial \PP} v\Big|
 \le C h_\calP^{1/2} \| \overline v-v \|_{0,\partial\calP}+\Big| \int_{\partial \PP} v\Big|\\
& \le C h_\calP^{1/2}\Big(h_{\PP}^{-1/2} \| \overline v-v \|_{0,\calP}+h_{\PP}^{1/2}|v|_{1,\PP}\Big)+\Big| \int_{\partial \PP} v\Big|\\
&\le C h_\calP | v |_{1,\calP} +\Big| \int_{\partial \PP} v\Big|.
\end{aligned}
\end{equation}
Inserting \eqref{due} in \eqref{uno} gives the result.
\end{proof}
\begin{cor}
An obvious consequence of \eqref{newnew5} is the Poincar\'e-type inequality
\begin{equation}\label{newnew6}
\| v \|_{0,\calP} \le C h_\calP | v |_{1,\calP} 
\qquad \forall v \in H^1(\calP) \textrm{ with } \int_{\partial\calP} v = 0 .
\end{equation}
\end{cor}
\begin{prop}\label{interp-prop}
Let assumption {\bf A} hold. Then there exists a positive constant $C$ such that for all $\calP \in \Th$ and any
function $u \in H^{s}(\calP)$, $s \ge 2$, 
it holds
\begin{equation}\label{stimau-uI}
| u - u_I |_{1,\calP} \le 
C h_\calP^{s-1} |u|_{s,\calP} \: .
\end{equation}
\end{prop}
\begin{proof}
In the following we assume that $s$ is integer, the general bound can then be immediately obtained by using classical results of space interpolation theory.
As a consequence of the properties of the map  ${\bf F}_\calP$, by mapping to $B_\calP$, applying standard estimates and finally mapping back, one can easily obtain (for all $0 \le \eps < 1/2$)
\begin{equation}\label{L:prel-1}
\tri v \tri_{1/2+\eps,\partial \calP} \le C \tri v \tri_{1+\eps,\calP} \quad \forall v \in H^{1+\eps}(\calP) \: .
\end{equation}
Let now ${\bf n}$ denote the outward unit normal to $\calP$ and let $\gamma$ be the curved edge of $\calP$.
For any $\varphib \in H({\rm div};\calP)$ we easily have, also using the second bound in \eqref{L:prel-2}, 
\begin{equation}\label{L:prel-3}
\begin{aligned}
\tri \varphib\cdot{\bf n} \tri_{-1/2,\partial\calP} 
& := \sup_{w \in H^{1/2}(\partial\calP)}\frac{<\varphib\cdot{\bf n},w>}{\tri w \tri_{1/2,\partial\calP}}
= \sup_{v \in H^1(\calP), \Delta v = 0} \frac{<\varphib\cdot{\bf n},v|_{\partial\calP}>}{\tri v|_{\partial\calP} \tri_{1/2,\partial\calP}} \\
& \le C \sup_{v \in H^1(\calP), \Delta v = 0} \frac{<\varphib\cdot{\bf n},v|_{\partial\calP}>}{\tri v \tri_{1,\calP}} \: ,
\end{aligned}
\end{equation}
where (here and in the sequel) the brackets denote a duality on the boundary of $\calP$.
Therefore an integration by parts and trivial inequalities 
(note that $h_E \| v \|_{0,\calP} \le \tri v \tri_{1,\calP}$)
yield the uniform bound
\begin{equation}\label{L:prel-4}
\tri \varphib\cdot{\bf n} \tri_{-1/2,\partial\calP} \le
C  \!\!\!\! \sup_{v \in H^1(\calP), \Delta v = 0}  \!\!\!\!
\frac{\int_E \varphib\cdot\nabla v + \int_E v ({\rm div}\varphib)}{\tri v \tri_{1,\calP}}
\le C \big( \| \varphib \|_{0,\calP} + h_{\calP} \| {\rm div} \varphib \|_{0,\calP} \big) \: .
\end{equation}
We can now estimate the interpolation error. 
Let the constant $c$ be equal to the average of $u-u_I$ on $\partial\calP$.
First by integrating by parts, then recalling the definition of $u_I$, we obtain
$$
\begin{aligned}
| u-u_I |_{1,\calP}^2 & = \int_\calP \nabla (u-u_I) \cdot \nabla (u-u_I-c) \\
& = - \int_\calP (u-u_I-c) \Delta (u-u_I) + <\nabla(u-u_I)\cdot{\bf n},u-u_I-c> \\
& \le C \| u-u_I-c \|_{0,\calP} \| \Delta u - \Pi^0_{k-2}(\Delta u) \|_{0,\calP} \\
& + \tri \nabla(u-u_I)\cdot{\bf n} \tri_{-1/2,\partial\calP}\, 
\tri u-u_I-c \tri_{1/2,\partial \calP} \: .
\end{aligned}
$$
We now apply standard polynomial approximation estimates on star-shaped domains and bound \eqref{L:prel-4} to derive
\begin{equation*}
\begin{aligned}
| u-u_I |_{1,\calP}^2 & \le C h_\calP^{s-1} | u-u_I |_{1,\calP} | \Delta u |_{s-2,\calP}
+ \big( | u-u_I |_{1,\calP} + h_{\calP} \| \Delta (u-u_I) \|_{0,\calP} \big) \tri u-u_I-c \tri_{1/2,\partial \calP} .
\end{aligned}
\end{equation*}
The above bound, again recalling the definition of $\Delta u_I$ and by standard approximation estimates for polynomials, becomes 
\begin{equation}\label{L:est-2}
\begin{aligned}
| u-u_I |_{1,\calP}^2 & \le C h_\calP^{s-1} | u-u_I |_{1,\calP} | u |_{s,\calP}
+ \big( | u-u_I |_{1,\calP} + h_{\calP}^{s-1} | u |_{s,\calP} \big) \tri u-u_I-c \tri_{1/2,\partial \calP} .
\end{aligned}
\end{equation}
Due to bound \eqref{L:est-2}, it is now easy to check by some simple algebra that the proof is concluded
provided we can show the boundary approximation estimate
\begin{equation}\label{L:target}
\tri u-u_I-c \tri_{1/2,\partial \calP} \le C h_{\calP}^{s-1} | u |_{s,\calP} \: .
\end{equation}
First by definition, then by a one dimensional Poincar\'e inequality (from $L^2$ into $H^{1/2}$, recalling that $c$ is the average of $u-u_I$ on $\partial\calP$) we get
$$
\tri u-u_I-c \tri_{1/2,\partial \calP} = h_{\PP}^{-1/2} \| u-u_I-c \|_{0,\partial \calP} + | u-u_I|_{1/2,\partial \calP} \le C | u-u_I|_{1/2,\partial \calP} \ .
$$
Since the function $u-u_I$ vanishes at all vertexes of $\calP$, a direct argument easily yields
$$
\| u-u_I\|_{0,\partial \calP} \le C h_\calP | u-u_I|_{1,\partial \calP}
$$
and thus, by space interpolation theory,
\begin{equation}\label{L:est-5}
| u-u_I|_{1/2,\partial \calP} \le C h_\calP^\eps | u-u_I|_{1/2+\eps,\partial \calP}
\le C \sum_{e \in \partial\calP} h_\calP^\eps | u-u_I|_{1/2+\eps,e}
\end{equation}
with $0 < \eps \le 1/2$ and with $e\in \partial\calP$ denoting all the edges (curved and straight) of $\calP$.

Whenever $e$ in \eqref{L:est-5} is a straight edge, we apply standard interpolation results in one dimension and a trace inequality (for instance from $e$ into the triangle connecting the endpoints of $e$ with the center of the ball $B_\calP$). We obtain
\begin{equation}\label{L:tar1}
h_\calP^\eps | u-u_I|_{1/2+\eps,e} \le h_\calP^{s-1} C |u|_{s-1/2,e}
\le C h_\calP^{s-1} C |u|_{s,\calP} \: .
\end{equation}
Whenever $e=\gamma$ in \eqref{L:est-5} is a curved edge, we first recall the definition of $u_I$ on $\gamma$ and apply the trace inequality in \eqref{L:prel-1}, then make use of standard polynomial approximation estimates on star-shaped domains
\begin{equation}\label{L:tar2}
h_\calP^\eps | u-u_I|_{1/2+\eps,\gamma} \le 
h_\calP^\eps \tri u-u_I \tri_{1/2+\eps,\partial \calP} 
\le h_\calP^{-1} \| u-q_k \|_{0,\calP} + h_\calP^\eps | u-q_k|_{1+\eps,\calP}
\le h_\calP^{s-1} | u|_{s,\calP} \: .
\end{equation}
The proof is concluded by combining \eqref{L:tar1}-\eqref{L:tar2} into \eqref{L:est-5}, thus obtaining \eqref{L:target}.
\end{proof} 

We have the following corollary.
\begin{cor}\label{interp-corol}
Under the same notation and assumptions of Proposition \ref{interp-prop} and its proof, by choosing $u_\pi=q_k$, for all elements $\calP$ it holds
\begin{equation}\label{stimaS}
\begin{aligned}
\| \gu_I - \gu_\pi \|_{1,S(\calP)}^2 & :=
a_h^{\PP}(\gu_I - \gu_\pi , \gu_I - \gu_\pi) 
\le C h_\calP^{2(s-1)} |u|^2_{s,\calP} \: ,
\end{aligned}
\end{equation}
and thus
\begin{equation}\label{stimaSh}
\| \gu_I - \gu_\pi \|_{1,S} \le C h^{s-1} |u|_{s,\Omega} .
\end{equation}
\end{cor}
\begin{proof}
We only show the sketch of the proof.
We start by noting that \eqref{stimau-uI} combined with a scaled Poincar\'e-type inequality easily yields
\begin{equation}\label{bound-L2}
\| u - u_I \|_{0,\calP} \le 
C h_\calP^{s} |u|_{s,\calP} \: .
\end{equation}
By definition, the norm $\| \cdot \|_{1,S}$ is split into a consistency term and a stabilization term
$$
\begin{aligned}
\| \gu_I - \gu_\pi \|_{1,S(\calP)}^2 & =
\int_{\PP} \pippof \nabla \Pi^{\nabla}_k(u_I - u_\pi)\cdot\nabla\Pi^{\nabla}_k(u_I - u_\pi)\,{\rm d}\PP \\
& + \S^{\PP}((\I-{\mathcal G}\Pi^{\nabla}_k)(\gu_I - \gu_\pi),(\I-{\mathcal G}\Pi^{\nabla}_k)(\gu_I - \gu_\pi)) .
\end{aligned}
$$
The estimate for the first term follows easily from \eqref{stimau-qk}, \eqref{stimau-uI} and the continuity of the $\Pi^\nabla_k$ operator.
Adopting definition \eqref{dofi-dofi}, the second term can be split into a part that involves the volume moments and a part that involves a sum of pointwise evaluations, including the $tgp$. Since the novelty here is in the second part (the first one is standard, as for straight polygons) we focus only on the latter term.
Let ${\mathcal N}$ denote the set of the boundary nodes of $\calP$ plus the tgp nodes. Recalling that $\gu_I$ and $\gu_\pi$ attain the same values on the $tgp$, we thus need to estimate (also using $u_I(\nu)=u(\nu)$ for all nodes $\nu \in {\mathcal N}, \nu \notin tgp$)
$$
\begin{aligned}
& \sum_{\nu \in {\mathcal N}} 
((\I-{\mathcal G}\Pi^{\nabla}_k)(\gu_I - \gu_\pi)(\nu))^2 = \sum_{\nu \in {\mathcal N}}((\gu_I - \gu_\pi)(\nu)-(\Pi^{\nabla}_k u_I - u_\pi)(\nu))^2\\
&\le 2\Big( \sum_{\nu \in {\mathcal N}}((\gu_I - \gu_\pi)(\nu))^2 +\sum_{\nu \in {\mathcal N}}(\Pi^{\nabla}_k u_I - u_\pi)(\nu))^2\Big)\\
& =
2  \sum_{\nu \in {\mathcal N}, \nu \notin tgp} ((u - u_\pi)(\nu))^2+
2 \sum_{\nu \in {\mathcal N}} ((\Pi^\nabla_k u_I - u_\pi)(\nu))^2 .
\end{aligned}
$$
The first term above is bounded by standard polynomial approximation estimates in $L^\infty$. In order to deal with the second term, let $\widetilde{B}$ be the smallest ball that contains $\calP$ and all the $tgp$; note that by definition of the $tgp$ the radius of such ball is uniformly comparable to that of $B_\calP$, see assumption {\bf A}.
As a consequence, from known properties of polynomial functions and an inverse estimate,
$$
\begin{aligned}
\sum_{\nu \in {\mathcal N}} ((\Pi^\nabla_k u_I - u_\pi)(\nu))^2 & \le
\|\Pi^\nabla_k u_I - u_\pi \|^2_{L^\infty(\widetilde{B})} \\
& \le C \|\Pi^\nabla_k u_I - u_\pi\|^2_{L^\infty(\calP)} 
\le C h_\calP^{-2} \|\Pi^\nabla_k u_I - u_\pi\|^2_{0,\calP} .
\end{aligned}
$$
Now the result follows by adding and subtracting $u$ and $\Pi^\nabla_k u$, using the triangle inequality, bound \eqref{bound-L2}, stability properties of $\Pi^\nabla_k$ and standard polynomial approximation estimates.
\end{proof}

\subsection{Construction of $u_I$ for discontinuous $\pippof\nn\cdot\nabla u$}

In the case where $u$ is continuous but not $C^1$ across $\gamma$ (that is, when $\pippof$ has two different values across $\gamma$), then we have to stabilize $a_h^{\calP}$ in only one of the two elements $\calP$ and $\calP^{\prime}$.
Let's say that we stabilize in $\calP$. Then the $q_k$ will
be chosen using only the values of $u$ in $\calP$ and the estimate \eqref{stimau-qk} will be required only in $\calP$. 
 
In $\calP^{\prime}$ we will choose $u_{\pi}$ as 
the $L^2(\calP^{\prime})$-projection of $u$ onto $\Po_k$. In turn, $u_I$ will still be equal to $q_k$ on $\gamma$ (it must be continuous) and will be defined, on the other edges different from $\gamma$, using the values of $u$ itself (and, inside, using  $\Delta u_I=\Pi_{k-2}^{0,\PP^\prime}(\Delta u$).

Using the fact that the estimate \eqref{stimau-qk} still holds on $\gamma$, we can deal with the estimate of 
$u-u_I$ proceeding as before. This will give us, intercalating $u$,   an estimate on 
$$\|u_I-u_{\pi}\|_{1,\PP^{\prime}}.$$ 
 Next, since we {\bf do not} stabilize the tgp degrees of freedom on 
$\calP^{\prime}$, this will also provide a bound for     
$$\|\gu_I-\gu_{\pi}\|_{1,S},$$ proceeding (for the other degrees of freedom) as for the usual VEMs
and we are done. We collect all this in the following proposition

\begin{prop}\label{onesided} If all elements of the decomposition satisfy assumption {\bf A}, using a one-sided stabilization
as described above we still have the estimate \eqref{stimaSh}
\end{prop}

\subsection{Estimate of ${\mathcal E}_1(f)$}

For $k=1$, using the definition \eqref{per-fh} and noting that for any elements $\calP$ and all $v \in H^1(\calP)$ it holds $\int_{\partial\calP} (v- \Pi^{\nabla}_1 v) = 0$ by definition, using bound \eqref{newnew6} we derive
\begin{equation}
(f,v-\T v)_0=\sum_{\PP}(f,v- \Pi^{\nabla}_1 v)_{0,\PP} \le C\sum_{\PP} \|f\|_{0,\PP} \, h_{\PP} |v- \Pi^{\nabla}_1 v|_{1,\PP} \le C h \|f\|_0 |v|_1,
\end{equation}
giving
\begin{equation}\label{Errfk1}
{\mathcal E}_1(f)\le C\,h \|f\|_{0} \qquad \mbox{for } k = 1.
\end{equation}
For $k\ge 2$, from the properties of the $L^2$-orthogonal projection and standard interpolation estimates we have 
\begin{align}
(f,v-\T v)_0&=\sum_{\PP}(f,v- \Pi^{0,\PP}_{k-2} v)_{0,\PP}= \sum_{\PP}(f-\Pi^{0,\PP}_{k-2} f,v- \Pi^{0,\PP}_{k-2} v)_{0,\PP} \\
&\le C\sum_{\PP} h^{k-1}_{\PP}\|f\|_{k-1,\PP}\, h_{\PP} |v|_{1,\PP}
\le C h^k (\sum_{\PP}\|f\|^2_{k-1, \PP})^{1/2} \|v\|_1,
\end{align}
implying
\begin{equation}\label{Errfk}
{\mathcal E}_1(f)\le C\,h^k (\sum_{\PP}\|f\|^2_{k-1, \PP})^{1/2} \qquad \mbox{for } k\ge 2.
\end{equation}

\subsection{The final error estimate}
At this point, we just have to join the error estimate  \eqref{ErrEst} with the results of this section, and in particular \eqref{stimau-uI}, \eqref{stimaSh} (taking also into account Proposition \ref{onesided})  and \eqref{Errfk1} or
\eqref{Errfk} we have  
\begin{equation}  
\|u-\fuh\|_1\le  Ch^k\Big(   (\sum_{\PP}  |u|_{k+1,\PP}^2 +  \|f\|^2_{s, \PP})^{1/2}\Big)^{1/2}
\end{equation}
where $s=1$ for $k=1$ and $s=k-1$ for $k\ge 2$.

\end{document}